\documentclass[12pt]{amsart}

\usepackage{amsmath, amssymb, amsthm, amsfonts, amscd}
\usepackage{mathrsfs}

\input xy
\xyoption{all}

\usepackage{hyperref}
\usepackage{graphicx}
\usepackage{color}

\numberwithin{equation}{section}

\newtheorem{theorem}[equation]{Theorem}
\newtheorem*{theorem*}{Theorem}
\newtheorem{corollary}[equation]{Corollary}
\newtheorem{lemma}[equation]{Lemma}
\newtheorem*{lemma*}{Lemma}

\newtheorem{proposition}[equation]{Proposition}
\newtheorem*{proposition*}{Proposition}
\newtheorem*{expectation*}{Expectation}
\newtheorem{conjecture}[equation]{Conjecture}

\theoremstyle{definition}

\newtheorem{definition}[equation]{Definition}
\newtheorem{remark}[equation]{Remark}
\newtheorem{example}[equation]{Example}
\newtheorem*{example*}{Example}
\newtheorem*{question*}{Question}

\usepackage[margin=1in,marginparwidth=0.8in, marginparsep=0.1in]{geometry}

\def\AA{\mathbb{A}}

\def\BB{\mathbb{B}}
\def\CC{\mathbb{C}}
\def\DD{\mathbb{D}}

\def\LL{\mathbb{L}}

\def\PP{\mathbb{P}}
\def\QQ{\mathbb{Q}}
\def\RR{\mathbb{R}}

\def\SS{\mathbb{S}}
\def\TT{\mathbb{T}}

\def\ZZ{\mathbb{Z}}

\def\cA{\mathcal{A}}

\def\cC{\mathcal{C}}

\def\cF{\mathcal{F}}
\def\cG{\mathcal{G}}

\def\cL{\mathcal{L}}

\def\cO{\mathcal{O}}
\def\cP{\mathcal{P}}

\def\cS{\mathcal{S}}
\def\cT{\mathcal{T}}

\def\cW{\mathcal{W}}
\def\cW{\mathcal{W}}

\def\bT{\mathbf{T}}

\newcommand\frL{\mathfrak{L}}

\newcommand\tils{\widetilde{s}}

\def\hatX{\widehat{X}}

\newcommand{\Coh}{\textup{Coh}}

\newcommand\id{\textup{id}}

\newcommand{\Ind}{\textup{Ind}}

\newcommand{\Perf}{\textup{Perf}}

\newcommand{\Pic}{\textup{Pic}}

\newcommand{\QCoh}{\textup{QCoh}}

\newcommand\Spec{\textup{Spec}}

\newcommand\Sym{\textup{Sym}}

\newcommand\Tot{\textup{Tot}}

\newcommand\Aut{\textup{Aut}}
\newcommand\Hom{\textup{Hom}}
\newcommand\End{\textup{End}}

\newcommand\nc{\newcommand}
\nc\on{\operatorname}
\nc\ol{\overline}
\nc\ul{\underline}

\nc\oo{\infty}

\nc\Cone{\mathit{Cone}}
\nc\ssupp{\mathit{ss}}
\nc\risom{\stackrel{\sim}{\to}}
\nc\Sh{\mathit{Sh}}
\nc\un{\diamondsuit}
\nc\orient{\mathit{or}}
\nc\sing{\mathit{sing}}
\nc\MF{\on{MF}}
\nc\Log{\on{Log}}
\nc\Arg{\on{Arg}}
\nc\inthom{\mathit{Hom}}
\nc\colim{\varinjlim}

\nc\LGr{\on{LGr}}
\nc\wmsh{\mu sh^{\cW}}
\nc\oX{\overline{X}}
\nc\Trop{\on{Trop}}
\nc\Conv{\on{Conv}}
\nc\Newt{\on{Newt}}
\nc\Def{\on{Def}}
\nc\dz{\on{d}\!\ZZ/2\on{g}}
\nc\pp{\PP}
\nc\Exc{\pp}
\nc\Gv{{G^\vee}}
\nc\G{G}
\nc\pcn{\partial\CC^n}
\nc\ppp{\partial\PP}
\nc\nsecc{\mathscr{S}'}
\nc\nsec{\mathscr{S}}
\nc\wsh{w\mu sh}
\nc\tpc{\widetilde{\partial\CC^n}}
\nc\tpn{\tpc}
\nc\tcn{\tpc}
\nc\Bl{\on{Bl}}
\nc\LLC{\mathring{\mathbb{L}}}
\nc\LLc{\mathring{\mathbb{L}}}
\nc\fc{\mathring{f}}
\nc\Pc{\mathring{P}}
\nc\Xc{\mathring{X}}
\nc\tg{\tilde{g}}
\nc\cPc{\mathring{\mathcal{P}}}
\nc\lol{\text{\reflectbox{$\multimap$}}}
\nc\ldl{\text{\reflectbox{$\multimap$}}}
\nc\Crit{\on{Crit}}
\nc\wdv{W_{\Delta^\vee}}
\nc\trt{\widetilde{\on{trop}}}
\nc\dtmir{\partial\bT^{mir}}
\nc\ttmir{\partial\bT^{mir}_{loc}}
\nc\FS{\on{FS}}
\nc{\Vc}{\mathring{V}}
\nc{\D}{\mathbb{D}}
\nc{\Dc}{\mathring{\mathbb{D}}}
\nc\GIT{/\!/}
\nc\todo[1]{{\color{red}\bf (#1)}}

\begin{document}


\title{Mirror symmetry for Berglund-H\"ubsch Milnor fibers}

\author{Benjamin Gammage}
\address{Dept. of Mathematics\\Harvard University\\1 Oxford St., Cambridge 02138}
\email{gammage@math.harvard.edu}

\begin{abstract}
	We explain how to calculate the Fukaya category of the Milnor fiber of a Berglund-H\"ubsch invertible polynomial, proving many cases of a
	conjecture of Yank{\i} Lekili and Kazushi Ueda on homological mirror symmetry. 
	As usual, we begin by calculating the ``very affine'' Fukaya category; afterwards, we deform it, generalizing an earlier calculation of David Nadler. The main step of our calculation may be understood as determining a certain canonical extension of a perverse schober.
\end{abstract}

\maketitle


\tableofcontents


\section{Introduction}\label{sec:intro}
Our story begins with a quasihomogeneous polynomial in $n$ variables with $n$ terms,
\begin{align}
  W:\CC^n\to\CC, && W(x_1,\ldots,x_n) = \sum_{i=1}^n\prod_{j=1}^nx_j^{a_{ij}}
\end{align}
where the exponent matrix $A= (a_{ij})$ is invertible in $GL(n,\QQ)$. In this case we say that $W$ is an {\em invertible} polynomial, and its {\em dual} polynomial $W^\vee$ is a polynomial of the same type,
\begin{equation}
  W^\vee(z_1,\ldots,z_n) := \sum_{i=1}^n\prod_{j=1}^nz_j^{a_{ji}}
\end{equation}
with exponents determined by the transpose matrix $A^t = (a_{ji}).$ These polynomials also define a Pontryagin dual pair of finite abelian groups $\G$ and $\Gv,$ to be defined below, acting on $\CC^n$ and for which $W$ and $W^\vee,$ respectively, are invariant.

The study of such pairs of polynomials was initiated in \cite{BH93}, who conjectured a mirror-symmetry relationship between them: the Landau-Ginzburg models $(\CC^n,W)$ and $(\CC^n/\Gv,W^\vee)$ (or, equivalently, $(\CC^n/G,W)$ and $(\CC^n, W^\vee),$ if we had a better understanding of $\G$-equivariant Floer theory) were expected to be dual in the sense of homological mirror symmetry. In the past two and a half decades, features of this mirror symmetry, especially the relevant cohomological field theories (of which the A-side theory is the FJRW theory introduced in \cite{FJR} and the genus-0 B-side theory is K. Saito's theory of primitive forms \cite{Saito1}) and their many relations to singularity theory, have been studied in many individual cases. A statement of homological mirror symmetry for such pairs, with all the equivariance on the B-side, may be stated as follows:
\begin{conjecture}\label{conj:bh}
  There is an equivalence
	\begin{equation}\label{eq:bh-statement}
  \cW(\CC^n,W)\cong \MF(\CC^n/\Gv,W^\vee)
\end{equation}
  between the wrapped Fukaya category of the Liouville sector associated to Landau-Ginzburg model $(\CC^n,W)$ and the category of $\Gv$-equivariant matrix factorizations of $W^\vee$ on $\CC^n$.
\end{conjecture}

This theorem ought to be proved\footnote{The approach described here draws on unpublished work with Jack Smith.} using the now-familiar strategy in mirror symmetry, pioneered by Seidel \cite{Sei-quartic,Sei-genus2} and systematized by Sheridan \cite{Sher1,Sher2,Sher3}, of studying the complement of a divisor on the A-side, proving mirror symmetry there and then deforming the resulting category by a count of holomorphic disks. In this case, this means beginning with the equivalence
\begin{equation}
  \cW((\CC^\times)^n, W)\cong \Coh(\CC^n/\Gv)
  \label{eq:fund-undef}
\end{equation}
and then finding $W^\vee$ as a count of disks passing through the divisor 
\[D=\{x_1\cdots x_n =0\}.\]

The fundamental fact which allows this trick to work is the observation that any invertible polynomial $W$ can be obtained from the basic case $W=x_1+\cdots+x_n$ by pulling back along the map
\begin{align}\label{eq:in-cover}
  \rho:\CC^n\to\CC^n, && 
  (x_1,\ldots,x_n)\mapsto \left( \prod x_j^{a_{1j}},\ldots,\prod x_j^{a_{nj}} \right),
\end{align}
which is a finite unramified cover away from the divisor $D$. The restriction of this map to $(\CC^\times)^n$ is thus a finite unramified cover everywhere, and by definition the group $\G$ is its Galois group. The general case can be proven through an understanding of this much simpler case.

We use the same fact in this paper, where our concern is not the ``singular category'' $\cW(\CC^n,W)$ but rather the ``nearby category,'' the wrapped Fukaya category of the Milnor fiber of $\CC.$ The study of this category was initiated in \cite{LU1}, and it has been described via explicit calculations in some special cases: for ADE singularities in \cite{LU2}, and for the 2-dimensional case in \cite{Hab} and \cite{CCJ}. As in \cite{Gam-Smith}, our strategy of proof in this paper allows us to treat all invertible polynomials in a uniform way.

Let $V=W^{-1}(n)$ be the Milnor fiber of $W$ in $\CC^n.$ (Choosing the fiber over $n$ as a representative Milnor fiber is an irrelevant normalization we find helpful.) We study $V$ through its ``very affine'' part $\Vc = W^{-1}(n)\cap(\CC^\times)^n.$ In analogy with the Landau-Ginzburg case, we begin with a mirror symmetry isomorphism
\begin{equation}\label{eq:VA}
\cW(\Vc)\cong \Coh(\partial\CC^n/\Gv)
\end{equation}
between the wrapped Fukaya category of $\Vc$ and the coherent sheaf category of the toric boundary 
\[\partial \CC^n/\Gv =\{z_1\cdots z_n=0\}/\Gv\]
of the toric stack $\CC^n/\Gv,$ as described in \cite{GS17,LP20}. As in the Landau-Ginzburg case, this isomorphism owes its existence to the cover \eqref{eq:in-cover}.

Finally, by understanding the geometry near the deleted divisor $V\cap \{x_1\cdots x_n=0\},$ we are able to represent $\cW(V)$ as a deformation of $\cW(\Vc)$. We state our result as conditional on Conjecture \ref{conj:main}: as we explain in \S\ref{subsec:conj}, this is an expected property of the Fukaya category, which in the setting at hand is possible to prove in many cases, but for which a complete account would be outside the scope of this paper.
\begin{theorem}\label{thm:main-intro}
	Assuming Conjecture \ref{conj:main}, the wrapped Fukaya category $\cW(V)$ is a deformation of $\Coh(\partial \CC^n/\Gv)$ by $W^\vee.$ In other words, there is an equivalence of 2-periodic dg-categories
	\begin{equation}\label{eq:thm-intro}
  \cW(V)\cong \MF^{\Gv}(\CC\times \CC^n, tx_1\cdots x_n + W^\vee).
\end{equation}
\end{theorem}

The conjecture \ref{conj:main} on which the theorem is conditional describes a way of reconstructing a Fukaya-Seidel category $\cW(X,f),$ for a superpotential $f$ whose only critical value is 0, from its restriction to $f^{-1}(\CC\setminus \{0\})$ together with its cap functor. The reconstruction described by Conjecture \ref{conj:main} is known when none of the critical points of $f$ are ``at infinity'' in the fiber over zero, or when all of them are; the conjecture asserts that this holds for intermediate cases as well. We highlight here examples of the extreme cases where Theorem \ref{thm:main-intro} is true unconditionally.

\begin{example}
	Let $W(x_1,\ldots,x_n) = \sum x_i^{a_i}.$ Then the equivalence of Theorem \ref{thm:main-intro} is a straightforward generalization of the main calculation of \cite{Nad-cnw}, as we explain in Section \ref{sec:nxings}. This calculation was the main inspiration for the present paper, whose goal is to popularize the technique used in \cite{Nad-cnw}.
\end{example}

\begin{example}
	Let $W(x,y) = x^2y + y^2x.$ Then the Milnor fiber $V$ is equal to its very affine part $\Vc$; as a result, the deformation by $W^\vee$ described in Theorem \ref{thm:main-intro} is trivial, and the right-hand side of \eqref{eq:thm-intro} is equivalent to the right-hand side of \eqref{eq:VA}. In other words, the second term in the superpotential on the right-hand side of \eqref{eq:thm-intro} has no effect on the category.
\end{example}

\begin{remark}\label{rem:zero}
	The previous example highlights a subtle complication of \eqref{eq:thm-intro} as opposed to the traditional Berglund-H\"ubsch mirror statement \eqref{eq:bh-statement}: namely, the fact that the superpotential $W^\vee$ may be zero everywhere along the toric boundary of $\CC^n/\Gv$; in this case, as we have just seen, the mirror to the Milnor fiber $V$ receives no contribution from $W^\vee$, and the category on the right-hand side of Equation \eqref{eq:thm-intro} is just $\Coh(\partial\CC^n/G^\vee)$.
\end{remark}

\subsubsection*{Comparing with \cite{LU1}}
The motivation for Theorem \ref{thm:main-intro} was the paper \cite{LU1}, which conjectured a $\ZZ$-graded version of the theorem. Traditionally in the Berglund-H\"ubsch theory, one considers a ``maximal symmetry group'' for the quasihomogeneous polynomial $W^\vee$: namely,
\[
	\Gamma^\vee:=\{(t_1,\ldots,t_n,t_{n+1})\in (\CC^\times)^{n+1}\mid 
	W^\vee(t_1z_1,\ldots, t_nz_n) = t_{n+1}W^\vee(z_1,\ldots,z_n)\}.
\]
Observe that $\Gamma^\vee$ has a natural map to $\CC^\times,$ given by projection to the last factor, and the kernel of this map is the finite group $G^\vee$ discussed above.
The main conjecture of \cite{LU1} is the following:
\begin{conjecture}[\cite{LU1}]\label{conj:lu}
	There is an equivalence of dg categories
	\begin{equation}\label{eq:lu-conj}
	\cW(V)\cong\MF^{\Gamma^\vee}(\CC_{x_0}\times \CC^n, x_0x_1\cdots x_n + W^\vee).
\end{equation}
\end{conjecture}
Although non-equivariant matrix factorization categories are 2-periodic, the right-hand side of \eqref{eq:lu-conj} can be understood as a usual ($\ZZ$-graded, not 2-periodic) dg category using the extra grading coming from the homomorphism $\Gamma^\vee\to \CC^\times.$ (The trick of using an extra $\CC^\times$-action to grade the category of matrix factorizations goes back to \cite{or-lgcy}.) If we only care about the $\ZZ/2$-graded category, we may forget this $\CC^\times$-action, retaining only the information of the finite symmetry group $G$. We conclude that after 2-periodicization, the right-hand side of \eqref{eq:lu-conj} becomes equivalent to $\MF^{G^\vee}(\CC\times \CC^n, x_0x_1\cdots x_n + W^\vee),$ which is the right-hand side of Theorem \ref{thm:main-intro}.

In other words, our Theorem \ref{thm:main-intro} is the 2-periodic version of Conjecture \ref{conj:lu}, obtained by forgetting the extra $\CC^\times$-action on the B-side and collapsing the Maslov index on the A-side. In order to prove Conjecture \ref{conj:lu} on the nose, one would have to equip the Milnor fiber $V$ with symplectic grading data --- namely, a trivialization of the bundle $(\Lambda^{\on{top}}_\CC TV)^{\otimes 2},$ or equivalently a map $\on{LGr}(V)\to K(\ZZ,1)$ whose restriction to each fiber represents the Maslov class --- such that the induced $\ZZ$-grading on the Fukaya category matches the $\ZZ$-grading on $\MF^{G^\vee}(\CC\times \CC^n, x_0x_1\cdots x_n + W^\vee)$ coming from the extra $\CC^\times$ action. (See Appendix \ref{sec:appendix} for further remarks on grading data.) This is likely not too difficult, but we omit it in this paper, whose aim is to highlight some other constructions in symplectic geometry.
\subsubsection*{Calculating the deformation}
In the proof sketch outlined above, the first step, the calculation of the Fukaya category $\cW(\Vc)$ of the very affine part, has already been accomplished in \cite{GS17}. As often happens, the more difficult step consists in relating this category to the Fukaya category of the partial compactification $\cW(V).$ 

However, we observe that, unlike in the situations normally considered by Seidel and Sheridan, we are studying only a {\em partial} compactification, and the space $V$ is itself a Weinstein manifold, so its Fukaya category can be approached using the locality and other techniques developed in \cite{Nad-wms,GPS1,GPS2,GPS3,NS20}. We will be inspired in particular by the calculations of \cite{Nad-cnw}, as we will explain below.

Since the partially compactified manifold $V$ remains Weinstein, the relation of $V$ and $\Vc$ may be studied within a neighborhood of the deleted divisor%
\footnote{It may seem more natural to write $D = V\cap\{x_1\cdots x_n=0\},$ but since we study $V$ through the cover \eqref{eq:in-cover}, it will be important that $D$ be the central fiber of the degenerating family \eqref{eq:intro-tfam}.}
\[D:=V\setminus \Vc = V\cap \left\{\prod_j x_1^{a_{1j}}\cdots x_n^{a_{nj}} =0\right\}.\]
As we will explain below, this can be understood as a question about the symplectic geometry of the degenerating family
\begin{equation}
	V\cap \left\{\prod_j x_1^{a_{1j}}\cdots x_n^{a_{nj}} = t\right\},
	\label{eq:intro-tfam}
\end{equation}
as $t$ varies in a small disk around 0. The results of \cite{GS17} give a decomposition of a smooth fiber in this family into simpler Liouville sectors, and we can analyze degenerations of those sectors individually. We find that the total space of the degeneration of each such sector is a copy of the Liouville sector studied in \cite{Nad-cnw}. That paper calculated the category of microlocal sheaves on the Lagrangian skeleton of this sector, and the papers \cite{GPS3,NS20} now allow us to understand the results of \cite{Nad-cnw} as a calculation in the wrapped Fukaya category.

\subsubsection*{Perverse schobers}
%

The key tool in understanding the deformation of categories discussed above is the map
\begin{align}\label{eq:xn-map}
	f:V\to \CC, && (x_1,\ldots,x_n)\mapsto \prod_j x_1^{a_{1j}}\cdots x_n^{a_{nj}}.
\end{align}
This map can be described very simply: it has only two critical values, one nondegenerate critical value and one degenerate critical value at 0. The questions of deformation theory take place around the latter; we want to relate the Fukaya categories of $V$ and $\Vc = V \setminus f^{-1}(0).$



In other words, we are interested in the behavior of the symplectic geometry of $V$ relative to the holomorphic map $f,$ an approach again pioneered by Seidel \cite{Sei-book} in the case of Lefschetz fibrations; the non-Lefschetz behavior of $f$ near its critical value 0 requires more sophisticated categorical tools. At the same time, the fact that the base $\CC$ of our fibration $f$ is itself Stein helps to simplify much of our analysis, allowing us to work locally on this base.


The appropriate language for encoding Fukaya-categorical structures relative to a base is the theory of {\em perverse schobers}, as developed in \cite{KS1,KS2,BKS}. A perverse schober on $\CC$ with two singularities, at the points $0,1\in\CC,$ is equivalent to the data of a diagram of categories
\begin{equation}\label{eq:schob-C}
\xymatrix{
\cP_{[0,1)}\ar[r]^-F &\cP_{(0,1)} &\ar[l]_-G\cP_{(0,1]},
}
\end{equation}
where $F$ and $G$ are spherical functors.
The notation is intended to suggest that $\cP_{\frL}$ is the category which $\cP$ assigns to the Lagrangian $\frL\subset \CC.$

In our situation, we can produce such a datum by taking a Liouville-sectorial cover of the base by ``left'' and ``right'' sectors, each containing one critical value, and lifting this decomposition to the total space. In this case, the central category $\cP_{(0,1)}$ will be equivalent to the Fukaya category of a general fiber $f^{-1}(\frac12),$ and $F,G$ are the spherical ``cap'' or ``boundary restriction'' functors from the Fukaya categories of the two Liouville sectors stopped at $f^{-1}(\frac12).$ 

The Fukaya category $\cW(V)$ should then be calculated as the ``global (co)sections'' of this perverse schober, defined as the homotopy colimit
\begin{equation}
  \cW(V)\cong \cP_{[0,1]}:=\colim\left(
  \xymatrix{\cP_{[0,1)}&\ar[r]^-{G^L}\cP_{(0,1)} \ar[l]_-{F^L}&\cP_{(0,1]}}\right),
  \label{eq:schob-glob}
\end{equation}
where $F^L,G^L$ are the left adjoints of $F$ and $G$. (In the language of Fukaya categories, these are the ``cup functors''\footnote{In the literature, these are also sometimes referred to as ``Orlov functors''; in this paper, we will prefer the name ``cup functor'' for its descriptiveness and for the relation to its adjoint, the ``cap functor.''} corresponding to the inclusion of $f^{-1}(\frac12)$ as a stop in the left- or right-hand Liouville sectors.)

Similarly, the category $\cW(\Vc)$ is computable from a perverse schober $\cPc$ on $\CC^\times$:
\begin{equation}
 \cW(\Vc)\cong \cPc_{\frL_{\CC^\times}}:=\colim \left(
  \xymatrix{\cPc_{\lol}&\ar[r] \cPc_{(0,1)} \ar[l]&\cPc_{(\frac12,1]}}\right)
  \label{eq:schob-glob-cs}
\end{equation}
where the Lagrangian $\frL_{\CC^\times}$ is the union of a small circle around 0 and a spoke connecting the circle to the point 1, and the Lagrangian $\lol$ is the union of a small circle around 0 and a spoke emanating to the right (but not reaching the other critical value 1).

In fact, given a sectorial cover as stipulated above, the equivalences \eqref{eq:schob-glob} and \eqref{eq:schob-glob-cs} are immediate from the \cite{GPS2} theory of sectorial codescent.

Building on prior work \cite{GS17} on mirror symmetry for hypersurfaces in $(\CC^\times)^n,$ we can identify all of the pieces in the latter decomposition. Let $\widetilde{\partial\CC^n}/G^\vee$ be the strict transform of $\partial \CC^n/G^\vee$ under the blowup of $\CC^n/G^\vee$ at the origin, and $\PP_{G^\vee}$ the exceptional divisor of the blowup, which meets $\tpc/G^\vee$ in its toric boundary $\partial\PP_{G^\vee}$.
\begin{lemma}[Lemma \ref{lem:colim-undef} below]
  The diagram \eqref{eq:schob-glob-cs} is equivalent to a diagram of the form
  \begin{equation}
    \Coh(\partial\CC^n/\G^\vee)\cong\colim
		\left(\xymatrix{\Coh^{\G^\vee}(\widetilde{\partial \CC^n})&\ar[l]\Coh(\partial\PP_{G^\vee}) \ar[r]&\Coh^{\G^\vee}(\{0\})}\right).
\end{equation}
\end{lemma}

The right and middle pieces in the diagrams \eqref{eq:schob-glob} and \eqref{eq:schob-glob-cs} are the same; the only difference is in the left-hand category. 
We would like to relate the category $\cP_{[0,1)}$ to the category $\cPc_{\lol}$ appearing in the perverse schober on $\CC^\times.$

\subsubsection*{Extending a perverse schober}
In other words, we would like to understand how $\cPc,$ which is a perverse schober on $\CC^\times,$ extends to a perverse schober on $\CC.$ Our hope is that the local category $\cP_{[0,1)}$ on $\CC$ can be recovered from the local category $\cPc_{\lol}$ on $\CC^\times,$ together with a small amount of extra data.

As motivation, consider the situation which perverse schobers are expected to categorify, namely the theory of perverse sheaves. This theory admits a 
gluing formalism, described in \cite{Verdier85,MV86,Beilinson87}; in the case of perverse sheaves on $(\CC,0),$ this gluing result (which can be found in various forms in op. cit.) reads as follows:
\begin{theorem}\label{thm:glue-sheaves}
	A perverse sheaf on $\CC$ may be reconstructed from a perverse sheaf $\cF$ on $\CC^\times$, a monodromic perverse sheaf $\cG$ on the normal cone to 0, and an equivalence between $\cG$ and the specialization of $\cF$ at 0.
\end{theorem}
We will not literally categorify the data of Theorem \ref{thm:glue-sheaves}, but we will take it as a suggestion that there should be a finite amount of ``extension data'' which can be used to reconstruct $\cP$ from $\cPc.$

The fundamental ingredient which we will use in this extension procedure is the cap-cup adjunction
\begin{equation}\label{eq:cap-cup}
	F:\cP_{[0,1)}\rightleftarrows\cP_{(0,1)}: F^L.
\end{equation}
As we mentioned above, this adjunction is {\em spherical}: concretely, this means that the category $\cP_{(0,1)}$ is equipped with an automorphism $\mu\in\Aut(\cP_{(0,1)})$, which we can understand geometrically as induced by the monodromy around 0 of the map \eqref{eq:xn-map}, and the monad
$T:=FF^L\in \on{Alg}(\End(\cP_{(0,1)}))$ 
admits a presentation as the cone
\begin{equation}\label{eq:monad2}
	T \cong \on{Cone}(\xymatrix{\mu^{-1}\ar^-{s}[r]&\id_{\cW(\cP_{(0,1)})}})
\end{equation}
of a certain natural transformation $s$ between the clockwise monodromy automorphism $\mu^{-1}$ and the identity functor of the nearby category $\cP_{(0,1)}.$ 

In \S\ref{subsec:conj}, we will explain how the element $s$ can be used to produce an element of $HH^0(\cPc_{(0,1)})$, which we denote by $\tils$; after collapsing to a $\ZZ/2$-grading, $\tils$ can be considered as an even-degree Hochschild element, which can be understood as a deformation class.
The following represents our proposal for a categorified analogue of Theorem \ref{thm:glue-sheaves}.

\begin{conjecture}\label{conj:schober-ext}
	After collapsing to $\ZZ/2$-gradings, the category $\cP_{[0,1)}$ is a deformation of $\cPc_{\lol}$ with deformation class $\tils.$
\end{conjecture}

In \S\ref{subsec:conj}, we state a precise version of Conjecture \ref{conj:schober-ext} in the case where the perverse schobers $\cP,\cPc$ come from the Fukaya categories of Liouville sectors, and we explain why this statement is expected to follow from standard properties of the Fukaya category. As explained in \cite[\S 1.3]{abouzaid-auroux}, in the case where $\cP$ is the perverse schober associated to the wrapped Fukaya category of a Landau-Ginzburg model $(X,f),$ the section $s$ underlying the cup-cap adjunction is a count of holomorphic sections of $f$ over a disk containing 0. These holomorphic disks are supposed to provide the deformation described in Conjecture \ref{conj:schober-ext}.

\subsubsection*{The basic calculation}
As we will explain in more detail in \S\ref{sec:nxings}, the inspiration for the calculations in this paper is the perverse schober described in \cite{Nad-cnw}, associated to the Landau-Ginzburg model $f:\CC^n\to \CC$ given by $f(x_1,\ldots, x_n) = x_1\cdots x_n.$ In fact, the fiber $\{x_1\cdots x_n=1\}$ in that case is enriched with an extra stop (an instance of the ``completed LG triple'' construction we describe in Definition \ref{defn:cutoff-triple}), so that the Fukaya category of the fiber equipped with this new sectorial structure is equivalent to the category $\Coh(\PP^{n-1})$ of coherent sheaves on $\PP^{n-1}.$

In \cite{Nad-cnw}, it is shown that the monodromy automorphism $\mu$ of this category is the functor of tensoring with the line bundle $\cO_{\PP^{n-1}}(-1),$ and the disk-counting natural transformation $s$, which by definition is a map
\[
	\cO_{\PP^{n-1}}(-1)\to \cO_{\PP^{n-1}},
\]
can equivalently be understood as a generic section $s\in \Gamma(\PP^{n-1},\cO(1)),$ which after rescaling coefficients can be written as $s[z_0,\ldots,z_{n-1}] = z_0+\cdots + z_{n-1}$.

Moreover, in this case, as shown in \cite[\S 5]{Nad-cnw}, the cap functor --- which in this case we may understand as a functor $\cW(\CC^n, x_1\cdots x_n)\to \Coh(\PP^{n-1})$ --- is {\em conservative}, which is related to the fact that $\CC^n$ is contractible, and any interesting Lagrangian objects in $\cW(\CC^n, x_1\cdots x_n)$ must restrict to interesting Lagrangians on the boundary of this sector, and therefore the cup-cap adjunction is monadic. From these facts one deduces the main theorem of \cite{Nad-cnw}, the identification of $\cP_{[0,1)} = \cW(\CC^n, x_1\cdots x_n)$ with the category of coherent sheaves on the zero locus of the section $s.$

To relate this to the discussion above, we may observe (see Lemma \ref{lem:colim-undef} below) that the category $\cPc_{\lol}$ in this case is equivalent to the category $\Coh(\on{Tot}(\cO_{\PP^{n-1}}(-1)))$ of coherent sheaves on the total space of the line bundle $\cO_{\PP^{n-1}}(-1).$ The section $s$ can be understood as a function on this line bundle, defining a class $\tils\in HH^0(\on{Tot}(\cO_{\PP^{n-1}}(-1))).$ As we recall in Example \ref{ex:cohasdef}, a result of Orlov identifies the category of coherent sheaves on $\{s=0\}$ with the category of matrix factorizations of $\tils$ on 
$\on{Tot}(\cO_{\PP^{n-1}}(-1)).$
We may therefore rephrase the main theorem of \cite{Nad-cnw} as an instance of Conjecture \ref{conj:schober-ext}:
\begin{theorem}
	For the perverse schober $\cP$ studied in \cite{Nad-cnw}, the category $\cP_{[0,1)}$ is a deformation of the category $\cPc_{\lol}$ by $\tils.$ 
\end{theorem}
We refer to \S\ref{sec:section4} for more details.

As we shall see, in this paper we shall encounter a calculation very similar to that in \cite{Nad-cnw}, but where the symplectic manifolds considered there have been replaced by $G$-covers; correspondingly, the mirror space $\PP^{n-1}$ considered above will be replaced by a $G^\vee$-quotient, a stack which we denote $\PP^{n-1}_\Gv.$ The natural transformation $s$ discussed above will now be a generic section of the line bundle%
\footnote{Technically for the main calculation of this paper we will consider not the line bundle $\cO(1)$ on $\PP^{n-1}_\Gv$ but its restriction to the toric boundary $\partial \PP^{n-1}_{\Gv}.$ As a result, the restriction of $s=W^\vee$ may be zero, in the situation described in Remark \ref{rem:zero}. In this case the cup-cap adjunction will no longer be monadic, but on the other hand the deformation by 0 is trivial, so monadicity is not required to understand it.}
$\cO_{\PP^{n-1}/\Gv}(1)$; up to scaling by $(\CC^\times)^n,$ $s$ is therefore the function $W^\vee$. We have now finally encountered the dual superpotential $W^\vee$ through mirror symmetry, thus fulfilling the purpose of this paper.

\subsubsection*{Future directions}
We have already suggested a natural framework for the deformation we consider, but we conclude with some further suggestions on how the ad hoc constructions of this paper might be encapsulated as part of a more systematic theory. 

The simplicity of the deformation theory involved in this paper is due largely to the fact that the coefficients of monomials in the deformation class $W^\vee$ are irrelevant, so long as they are nonzero -- and moreover, for grading reasons, these monomials are the only ones which can appear. (The same phenomenon occurs also in the similar situation treated in \cite{Gam-Smith}.)

This is in contrast to the situations usually considered by Seidel and Sheridan, where one is ultimately interested in the Fukaya category of a {\em compact} symplectic manifold. The situation here has been simplified considerably because the total space $V$ is Stein and the fiber $D=V\setminus \Vc$ we delete is also Stein. The coefficients in the deformation class should be sensitive to the symplectic area of the deleted divisor $D,$ but since $D$ is not compact and does not have a well-defined symplectic area, we are free to scale these coefficients however we like. It would be useful to have a more direct exposition of how rescalings of Liouville structure in such a situation affect the deformation class.

Also, we note that in the situation considered in this paper, none of the coefficients of $W^\vee$ was zero. This follows from an explicit check conducted in \cite{Nad-cnw}, but it would be more satsfying to have a general criterion for when this occurs. From the perspective of the base, this should fit into a more general theory of extensions of perverse schobers, categorifying the theory of extensions of perverse sheaves.

Finally, we note that although we were able to reproduce the data of a perverse schober as described in \eqref{eq:schob-C} using the theory of Liouville sectors, it would be more satisfying to see this structure directly at the level of skeleta: for a perverse schober $\cP$ obtained as pushforward of a Fukaya category along a map $f$ with Weinstein fibers, the category which $\cP$ assigns to a Lagrangian $\frL$ in the base should be the category associated to a certain lift $\widetilde{\frL}$ of $\frL$ to a Lagrangian in the total space. Heuristically, $\widetilde{\frL}$ is obtained from the Lagrangian skeleton $\LL$ of a fiber by parallel transport over $\frL,$ collapsing vanishing cycles in $\LL$ when $\frL$ meets a critical value. 
Such a theory can be implemented ``by hand'' in the case when $f$ is a Lefschetz fibration with base $\CC$; for more general singularities, or higher-dimensional bases (as considered for instance in \cite{GMW}), a good general theory does not yet exist and requires the development of a better understanding of how to lift Liouville structures.


%
%
%

\subsubsection*{Conventions}
Throughout this paper, we work with pretriangulated dg-%
categories over the field $\CC,$ in the homotopical context of derived Morita theory, as described in \cite{Toen-morita},
or equivalently stable $\CC$-linear $\infty$-categories.
In particular, by $\Coh(X)$ or $\cW(X)$ we always mean the corresponding pretriangulated dg category, and all limits, colimits, and equivalences among these should be understood in the appropriate homotopical sense. 

We will also want to collapse the $\ZZ$-grading on dg categories to a $\ZZ/2$-grading, or equivalently to work over the field $\CC((\beta))$ where $\beta$ is a degree-2 variable.
(See \cite[\S 5]{Dyck-compactgen} for the adaptation of the above homotopical formalism to the 2-periodic case.)
For a dg category $\cC,$ we write $\cC_{\ZZ/2}$ for the 2-periodic dg-category obtained from base change along the map
$\CC\to \CC((\beta)).$

To simplify calculations, throughout this paper the wrapped Fukaya category $\cW(X)$ of a symplectic manifold or Liouville sector will always be taken with 2-periodic coefficients. In Appendix \ref{sec:appendix} we recall the usual grading data used to define the Fukaya category and explain the simplifications which occur in the $\ZZ/2$-graded case.

{\em Notation:} In this paper, we will denote the Milnor fiber we study by $V$ and its ``very affine'' open subset by $\Vc.$ (In the basic case where $W(x_1,\ldots,x_n) = x_1 + \cdots + x_n,$ we denote these spaces instead by $P$ and $\Pc.$) We write $\partial\CC^n/\Gv$ for the toric boundary of the toric stack $\CC^n/\Gv$ (which, as we shall see, will be the mirror to $\Vc$). We will also be interested in the blowup $\tcn/\Gv$ of the stack $\CC^n/\Gv$ at 0, and its exceptional divisor, which we denote by $\pp_\Gv.$ We will also find it useful to use the presentation of this blowup as the total space of the line bundle $\cO_{\pp_\Gv}(-1).$ Finally, we will also study the toric boundary $\tpn/\Gv$ of the blowup,  which we will think of as the total space of the restricted line bundle $\cO_{\ppp_\Gv}(-1).$
\subsubsection*{Acknowledgements}
I am grateful to Jack Smith for his collaboration on the shared work \cite{Gam-Smith}, of which this paper is an offshoot. I would also like to thank David Nadler, John Pardon, Vivek Shende, and Nick Sheridan for helpful conversations and advice, and Maxim Jeffs for sharing his thoughts on splitting Liouville sectors over an LG base.
I also acknowledge the support of an NSF postdoctoral fellowship, DMS-2001897.
\section{Geometric and categorical background}
Here we collect some results from mirror symmetry, toric and tropical geometry which will be necessary for some of the calculations in this paper. \S \ref{sec:gps} is a review of wrapped Fukaya categories of Liouville sectors and their computation, and in \S \ref{sec:veryaffine} we recall some features of the symplectic geometry of hypersurfaces in $(\CC^\times)^n$ and their toric mirrors. Homological mirror symmetry equivalences for these spaces were established in \cite{GS17} using the tropical methods of Mikhalkin \cite{M}, and we will recall those tropical methods as well.


\subsection{Microlocal sheaf methods}\label{sec:gps}
We begin by reviewing some of the results of \cite{GPS1,GPS2,GPS3,NS20} on calculation of Weinstein Fukaya categories using microlocal sheaves. The key feature of Weinstein symplectic geometry which makes computations tractable is {\em locality}: unlike in compact symplectic geometry, the symplectic behavior of a Weinstein manifold can be reconstructed from an open cover by simpler pieces, the {\em (Weinstein) Liouville sectors} defined in \cite{GPS1}. These are Weinstein manifolds with boundary which are appropriate for Weinstein gluings along their shared boundary. We refer to \cite{GPS1} for details on the definition and technical properties of Weinstein sectors, and we will summarize here the ways in which Weinstein sectors arise for us.

\begin{definition}
	Let $(\oX,\omega=d\lambda)$ be a Liouville domain with boundary $\partial \oX$ and completion to a Liouville manifold $X$. We assume that $X$ has in addition a {\em Weinstein} structure, namely, a function $f:X\to \RR$ for which the Liouville flow is gradient-like. (We will never discuss non-Weinstein Liouville manifolds in this paper.)
  \begin{enumerate}
    \item Let $\overline{F}\subset\partial\oX$ be a real hypersurface with boundary such that $(\overline{F},\lambda)$ is itself a Weinstein domain. This is a {\em Weinstein pair} in the sense of \cite{Eliashberg-revisited}, and we denote by $(X,F)$ the Liouville sector obtained from $\oX$ by completing away from a standard neighborhood of $F.$
		\item Let $\Lambda\subset \partial\oX$ be a closed Legendrian (possibly singular, so that by ``Legendrian'' we mean that $\Lambda$ has a smooth Legendrian submanifold $\Lambda^\circ$ whose complement is of strictly lower dimension). This is a {\em stop} in the sense of \cite{Sylvan}, and we write $(X,\Lambda)$ for the Liouville sector obtained by completing away from a standard neighborhood of $\Lambda.$
		\item Let $f:X\to \CC$ be a Liouville Landau-Ginzburg model in the sense of \cite[Example 2.20]{GPS1}: namely, we require that one can choose defining Liouville domains $\oX$ for $X$ and $\overline{F}$ for a generic fiber $F=f^{-1}(z)$ such that $\overline{F}$ is contained in the contact boundary $\partial \oX$ of $\oX.$ (A construction of such $\overline{F}, \oX$ in the case where $f$ is a polynomial function on algebraic variety $X$ can be found in \cite[Proposition 1]{Jeffs}.) Then we write $(X,f)$ for the corresponding Liouville sector.
  \end{enumerate}
\end{definition}
The above constructions are all essentially equivalent; for instance, from the third, one can obtain the first by taking $F$ to be a general fiber of $f,$ and the second by taking $\Lambda$ to be the skeleton of $F$. Recall that the {\em skeleton} $\LL_X$ of a Liouville manifold $X$ is the set of points in $X$ which do not escape to infinity under the Liouville flow, and the skeleton (or {\em relative skeleton}) of a Liouville sector $(X,\Lambda)$ is the set of points which do not escape to the complement of $\Lambda$ under Liouville flow. In other words, the relative skeleton of $(X,\Lambda)$ is the union of the skeleton $\LL_X$ of $X$ with the cone (under Liouville flow) of $\Lambda.$

One way to understand the skeleton $\LL_X$ of a Weinstein manifold $X$ is as recording gluing data describing a cover of $X$ by simpler Weinstein sectors. If $(X_1,F_1)$ and $(X_2,F_2)$ are two Weinstein pairs equipped with an isomorphism $F_1\cong F_2,$ and we write $\Lambda\subset F$ for the skeleton of this Weinstein manifold, then the Weinstein gluing $X_1\cup_F X_2,$ for which local models can be found in \cite[\S 3.1]{Eliashberg-revisited} and \cite[\S 2.6]{AGEN3}, will have glued skeleton $\LL_{(X_1,F_1)}\cup_\Lambda\LL_{(X_2,F_2)}.$

One source of such gluing presentations is a splitting over the base of a Landau-Ginzburg model.
\begin{example}
	Let $f:X\to \CC$ be a Liouville Landau-Ginzburg model, and let $(a,b)\subset \RR$ be an interval so that the strip $H=\{z\in\CC\mid \Re(z)\in(a,b)\}$ does not contain any critical values of $f$ and such that the Liouville structure on the preimage $f^{-1}(H)$ is a product Liouville structure for the presentation $f^{-1}(H)=F\times T^*(a,b),$ where $F$ is a fiber of $f.$
	Then $X$ has a presentation as a pair of sectors $X_L$ and $X_R$ glued along sector $F\times T^*[a,b].$
\end{example}

A generalization of sectorial gluings allowing for higher-codimensional strata is contained in \cite[\S 9.3]{GPS2} in terms of {\em Liouville sectors with (sectorial) corners}: these sectors arise naturally when one tries to perform a variant of the above gluing construction where sectors are glued not along boundary Weinstein manifolds but rather along boundary Weinstein sectors.
Cornered Liouville sectors often arise naturally via the following construction.
\begin{definition}\label{defn:cornered-lg}
	Let $f_1,f_2:X\to \CC$ be a pair of Liouville Landau-Ginzburg models with the same underlying space $X,$ and 
	write $F_1,F_2$ for general fibers of the functions $f_1, f_2$ and $F_{12}$ for a general fiber of the function $(f_1,f_2):X\to \CC^2.$ 
	Assume that $\Crit(f_1|_{F_2})\subset \Crit(f_2),$ i.e., that restriction to $F_2$ does not introduce any new critical points on $f_1.$ Then we call $(X,f_1,f_2)$ an {\em LG triple}, and we associate to it the Liouville sectorial structure on $X$ with stop given by the glued Liouville sector $(F_1,f_2)\cup_{F_{12}}(F_2,f_1).$
\end{definition}

The reason for the asymmetry in the above definition is that we would like to think of $f_1:X\to \CC$ as the fundamental structure in an LG triple, where we have enhanced the fibers of $F_1$ with the sectorial structure given by $F_2.$ In general, functions $f_1,f_2$ may not satisfy the hypotheses of Definition \ref{defn:cornered-lg}, but that can be fixed by passing to a completion, as in the following construction.
\begin{definition}\label{defn:cutoff-triple}
	Let $f_1,f_2:X\to \CC$ be a pair of Liouville LG models on $X$. Let $\BB\subset \CC$ be a disk such that $\BB$ contains all critical values of $f_1,$ and all new critical values of $f_1|_{F_2}$ lie outside of $\BB.$ Then we write $\hatX:=f_1^{-1}(\BB)$ for the $f_1$-preimage of $\BB,$ and $(\hatX,f_1,f_2)$ for the resulting LG triple in the sense of Definition \ref{defn:cornered-lg}. We call $(\hatX, f_1,f_2)$ a {\em completed LG triple}.
\end{definition}

\begin{remark}
	LG triples (and generalizations with more superpotentials) play a fundamental role in the work \cite{abouzaid-auroux} of Abouzaid-Auroux on mirror symmetry for hypersurfaces in $(\CC^\times)^n.$ One may alternatively perform the cutoff construction of Definition \ref{defn:cutoff-triple} by replacing $f_2$ with $f_2^\epsilon:=\epsilon f_2$; as $\epsilon\to 0$, the new critical values of $f_1$ on the preimage $(f_2^\epsilon)^{-1}(p)$ of a fixed point $p$ in $\CC$ will move far from $0\in \CC.$ The constructions of Definitions \ref{defn:cornered-lg} and \ref{defn:cutoff-triple} are not strictly necessary for this paper, but they will be helpful in conceptualizing the relation of our constructions to the calculations of \cite{Nad-cnw}.
\end{remark}

The main achievement of \cite{GPS2} consists in using cornered Liouville sectors to establish a local-to-global principle for calculation of wrapped Fukaya categories of Liouville manifolds.
The wrapped Fukaya category $\cW(\cS)$ of a Liouville sector $\cS$ is defined in \cite{GPS1}. If the sectorial structure $\cS=(X,\Lambda)$ comes from a stop, this can be understood as a partially wrapped Fukaya category; if $\cS=(X,f)$ is a Landau-Ginzburg sector, then $\cW(\cS)$ can be understood as a Fukaya-Seidel type category. The main results of \cite{GPS1,GPS2} are the following:
\begin{theorem}[\cite{GPS1,GPS2}]\label{thm:codescent}
  The wrapped Fukaya category $\cW$ is covariant for inclusions of Liouville sectors: if $\cS'\hookrightarrow\cS$ is an inclusion of subsectors, then there is a functor $\cW(\cS')\to \cW(\cS).$ Moreover, $\cW$ satisfies codescent along sectorial covers: if $X$ is a Weinstein manifold or sector which admits a cover $X= U_1\cup\cdots\cup U_n$ by subsectors $U_i$, then the natural map
  \[
  \colim_{\emptyset\neq I\subset [n]}\cW\left( \bigcap_{i\in I}U_i \right)\to\cW(X)
  \]
  from the homotopy colimit of wrapped Fukaya categories of the $U_i$ to the wrapped Fukaya category of $X$ is an equivalence of dg categories.
\end{theorem}

This theory heavily reduces the difficulty of computation of the wrapped Fukaya category $\cW(X)$ -- it remains only to compute the (hopefully much simpler) categories associated to subsectors $U_i.$ In this paper, we will barely need to do those calculations, since we will ultimately reduce to sectors $U_i$ whose Fukaya categories have already been calculated. And these calculations can often be performed simply using the language of microlocal sheaves, as developed in \cite{KS94}. In fact, we will find that the microlocal-sheaf-theoretic calculations of interest to us have already been performed in \cite{GS17, Nad-cnw}. We now recall the main theorems comparing these calculations to wrapped Fukaya categories.

\begin{theorem}[\cite{Sh-hp,NS20}]\label{thm:global-mush}
  Let $X$ be a stably polarized
  Weinstein sector with skeleton $\LL$. There exists a cosheaf of dg categories $\wmsh$ on $\LL.$ If $U\subset X$ is an open subset with an exact equivalence $U\cong T^*M$ for some manifold $M$, equipped with the cotangent fiber polarization, then $\wmsh$ agrees with the ``wrapped microlocal sheaves cosheaf'' defined in \cite{Nad-wms}.
\end{theorem}
In Appendix \ref{sec:appendix}, we recall the polarization data necessary to define $\wmsh$ and the wrapped Fukaya category, and then we explain why working with $\ZZ/2$-graded dg categories renders the precise choice of polarization data irrelevant; we will therefore mostly suppress discussion of polarization data in the remainder of the paper.
\begin{theorem}[{\cite[Theorem 1.4]{GPS3}}]\label{thm:gps-maincomparison}
	For $X$ a stably polarized Weinstein sector with skeleton $\LL$ as above, then there is an equivalence of categories $\wmsh(\LL)^{op}\cong \cW(X)$ between the opposite of the wrapped microlocal sheaves category on $\LL$ and the wrapped Fukaya category of $X$.
\end{theorem}

Finally, we highlight one feature of the functoriality for Liouville sectors which plays a major role in this paper. 
\begin{definition}
Let $\cS=(X,F)$ be a Weinstein Liouville sector, thought of as a Weinstein pair. The inclusion of the subsector $F\times T^*\RR$ into $\cS$ determines a functor
\begin{equation}\label{eq:or-fun}
\cup:\cW(F\times T^*\RR) = \cW(F)\to \cW(\cS),
\end{equation}
which is the {\em (Orlov) cup functor} associated to the pair $(X,F).$ Its right adjoint 
  \[
  \cap:\cW(\cS)\to \cW(F)
  \]
	s the \em{cap functor}.
\end{definition}

A general theory of the cup and cap functors can be found in \cite{Sylvan-orlov}. The most important feature of these functors is that they are {\em spherical}, with twist given by the monodromy automorphism. In other words, we have the following:
\begin{theorem}[\cite{Sylvan-orlov}]\label{thm:sylvan}
	Let $(X,F)$ be a Weinstein Liouville sector which can be presented as a Landau-Ginzburg model $(X,f),$\footnote{This hypothesis enforces ``swappability,'' a technical condition in \cite{Sylvan-orlov} which will be satisfied by all the Liouville sectors we consider.} and let $\mu^{-1}\in \Aut(\cW(F))$ be the automorphism of the wrapped Fukaya category of the fiber induced by clockwise monodromy in the base of $f.$ Then the monad $\cap\cup$ of the cup-cap adjunction admits a presentation
	\begin{equation}
		\cap\cup = \Cone(\mu^{-1}\xrightarrow{s} \id_{\cW(F)})
	\end{equation}
	as the cone on a natural transformation $s:\mu^{-1}\to \id_{\cW(F)}.$
\end{theorem}

\begin{example}\label{ex:orlov-triple}
	Let $(X,f_1,f_2)$ be the cornered Weinstein Landau-Ginzburg sector coming from an LG triple as in Definition \ref{defn:cornered-lg}. Then the wrapped Fukaya category $\cW(X,f_1,f_2)$ receives Orlov functors from the Fukaya categories of its boundaries, the Landau-Ginzburg sectors $(F_1, f_2)$ and $(F_2, f_1),$ and these in turn receive Orlov functors from the Fukaya category of their shared boundary, the Weinstein manifold $F_{12}.$ In other words, we have a square of spherical functors
	\[
		\xymatrix{
			\cW(F_{12})\ar[r]\ar[d]& \cW(F_1, f_2)\ar[d]\\
			\cW(F_2,f_1)\ar[r] & \cW(X,f_1,f_2).
		}
	\]
	Moreover, by the Weinstein codescent Theorem \ref{thm:codescent} and the sectorial gluing $\partial(X,f_1,f_2) = (F_1,f_2)\cup_{F_{12}}(F_2,f_1),$ we see that the homotopy pushout of the two functors with domain $\cW(F_{12})$ is the category $\cW(\partial (X,f_1,f_2)),$ the wrapped Fukaya category of the total boundary sector of $(X,f_1,f_2).$
\end{example}

\subsection{Very affine hypersurfaces}\label{sec:veryaffine}
We will also need some results from \cite{GS17} on skeleta of hypersurfaces in $(\CC^\times)^n.$ In fact, we will only be interested in two such hypersurfaces (along with their $\G$-covers): the pants, and the mirror to the boundary of projective space. We begin by recalling the abstract setup, 
and then we specialize to the cases of interest.

\begin{definition}
  Let $N$ be an $n$-dimensional lattice. To $N$ we associate the following spaces:
  \begin{itemize}
    \item The $n$-dimensional real vector space $N_\RR=N\otimes\RR$;
    \item The $n$-torus $N_{S^1} = N_\RR/N$;
    \item The cotangent bundle $T^*N_{S^1} = N_\RR/N\times N^\vee_\RR$ of this torus, whose projections to the base and fiber we denote by $\Arg$ and $\Log,$ respectively;
    \item The $n$-complex-dimensional split torus $N_{\CC^\times}=N\otimes\CC^\times.$
  \end{itemize}
      We also choose once and for all
      an inner product on $N$ in order to identify the cotangent bundle $T^*N_{S^1}$ with the complex torus $N_{\CC^\times}$ (which is more naturally the {\em tangent bundle} of $N_{S^1}$).
\end{definition}

An algebraic hypersurface in $N_{\CC^\times}$ determines a Newton polytope $\Delta\subset N_\RR^\vee,$ the convex hull of the monomials in a function defining $N_{\CC^\times}$ (under the standard identification of characters of $N_{\CC^\times}$ with lattice points in $N^\vee$), and we require that $0\in \Delta.$ We also choose a triangulation $\cT$ of this Newton polytope 
whose vertices are $\text{vertices}(\Delta)\cup\{0\}.$ 
For the examples in this paper, we will always take $\cT$ to be the triangulation whose simplices are the cones on the faces of $\Delta$ (and $0$).
We can thus interpret $\cT$ equivalently as the fan $\Sigma$ of cones on the faces of $\Delta.$

Inside the cotangent bundle $\TT,$ we specify a Lagrangian as follows:
\begin{equation}
  \LL_\Sigma:= \bigcup_{\sigma\in \Sigma} \overline{\sigma^\perp}\times \sigma\subset (N_\RR/N)\times N_\RR^\vee = T^*N_{S^1}.
  \label{eq:bondal-lag}
\end{equation}
The Lagrangian $\LL_\Sigma$ was first studied by Bondal in \cite{Bo}, then later studied extensively by Fang-Liu-Treumann-Zaslow \cite{FLTZ1,FLTZ2,FLTZ3}, for the relation between the category $\Sh_{-\LL_\Sigma}(N_{S^1})$ of constructible sheaves on $N_{S^1}$ microsupported along $-\LL_\Sigma$ and the category $\QCoh(\bT_\Sigma)$ of quasi-coherent sheaves on the toric stack $\bT_\Sigma$ with fan $\Sigma.$ Following the above works, a complete statement was first obtained in \cite{Ku} (followed by another proof in many cases in \cite{Zhou-ccc}). In modern language, the best statement reads as follows:
\begin{theorem}[\cite{Ku}]
  There is an equivalence \begin{equation}
    \wmsh(\LL_{\Sigma})\cong\Coh(\bT_\Sigma)
    \label{thm:ccc}
  \end{equation} between the dg category of wrapped microlocal sheaves on $\LL_\Sigma$ and the category of coherent sheaves on the toric variety $\bT_\Sigma.$ 
\end{theorem}

This equivalence was explained in \cite{GS17} by relating $\LL_\Sigma$ to the Landau-Ginzburg model $(N_{\CC^\times},W_{HV,\Sigma}),$ where the Hori-Vafa superpotential $W_{HV,\Sigma}$ traditionally understood as the mirror to the toric stack $\bT_\Sigma$ is a Laurent polynomial with Newton polytope $\Delta.$

\begin{theorem}[\cite{GS17}]\label{thm:gs-skel-1}
  Suppose that the fan $\Sigma$ is simplicial and that all generators of rays in $\Sigma$ lie on the boundary of the polytope $\Delta.$ Then the Lagrangian $\LL_\Sigma$ is the skeleton of the Landau-Ginzburg Liouville-sectorial structure defined on $N_{\CC^\times}$ by the function $W_{HV,\Sigma}:N_{\CC^\times}\to\CC.$
\end{theorem}

Let $\Lambda_\Sigma = \LL_\Sigma^\infty$ be the Legendrian boundary of the conic Lagrangian $\LL_\Sigma.$
The above theorem follows from the following calculation, first performed for $W_{HV} = 1 + x_1+\cdots x_n$ in \cite{Nad-wms} and generalized to a global statement in \cite{GS17} (and partially expanded and corrected in \cite{Zhou-skel}):
\begin{theorem}[\cite{GS17,Zhou-skel}]\label{thm:gs-skel}
  With hypotheses as in \ref{thm:gs-skel-1}, the Legendrian $\Lambda_\Sigma$ is a skeleton for the hypersurface $H_\Sigma=\{W_{HV,\Sigma}=0\}$, which embeds as a Weinstein hypersurface inside the contact boundary of $T^*N_{S^1}.$
\end{theorem}

\subsubsection{Tropical geometry}\label{sec:trop}
  We will find it useful to recall an outline of the proof for the above theorem, although we refer to \cite{GS17} for details. The main ingredient is {\em tropical geometry}: we refer to \cite{M} for details but recall here that given a hypersurface $H\subset N_{\CC^\times}$ whose Newton polytope $\Delta$ is equipped with triangulation $\cT,$ one can define a PL complex $\Trop(H)\subset N_\RR^\vee,$ the {\em tropicalization} of $H,$ with the following properties:
  \begin{itemize}
    \item The {\em amoeba} $\Log(H)$ is ``near'' to $\Trop(H),$ in a sense to be discussed below.
    \item The complex $\Trop(H)$ is dual to the triangulation $\cT$; in particular, if the triangulation $\cT$ comes from a complete fan $\Sigma$, then the complement $N_\RR^\vee\setminus \Trop(H)$ will have only one bounded component.
    \item The components of the complement $N_\RR^\vee\setminus \Trop(H)$ correspond to monomials in a defining equation for $H,$ with the corresponence associating each region to the monomial dominating there. 
  \end{itemize}

  As we require that the triangulation $\cT= \Sigma$ be simplicial, we can reduce to the case of $\Delta=\Delta_{std}$ the basic simplex in $N_\RR^\vee,$ so that $H=\Pc$ is the $(n-1)$-dimensional pants, and the skeleton of this variety was calculated in \cite{Nad-wms}. The main technical tool involved in the calculation is a symplectic isotopy of $H$, which we call ``tailoring,'' described first in \cite{M} and studied in detail in \cite{A1}:

  \begin{proposition}[{\cite[\S 4]{A1}}]\label{prop:tailoring}
    There is a symplectic isotopy $\{H^s\}_{0\leq s\leq 1}$ given by
		\[ H^s = \left\{\sum_{\alpha\in \Delta\cap N^\vee}(1-s\phi_\alpha(\Log(z)))c_\alpha z^\alpha=0\right\},
    \]
    where $c_\alpha$ are constants and the function $\phi_\alpha:N_\RR^\vee\to \RR$ has the property that near a face $F$ of $\Trop(H),$ $\phi_\alpha\equiv 1$ unless the monomial $c_\alpha z^\alpha$ dominates in a region of $N_\RR^\vee\setminus \Trop(H)$ adjacent to $F$.
  \end{proposition}
  In other words, near a $k$-face $F$ in $\Trop(H)$ dual to a standard $(n-k)$-simplex in $\cT$, the tailored hypersurface $H^1$ is equal to a product $(\CC^\times)^k\times \Pc_{n-1-k}$ of $(\CC^\times)^k$ with an $(n-1-k)$-dimensional pants. (If $F$ is dual to a larger $(n-k)$-simplex in $\cT,$ the second factor in this product will be replaced by an abelian cover.) This means in particular that sufficiently far in the interior of a $k$-face $F$, the amoeba $\cA:=\Log(H^1)$ of the tailored hypersurface agrees with $\Trop(H)$ in the directions tangent to $F$.

  Now the strategy of proof of Theorem \ref{thm:gs-skel} can be roughly understood as follows. Draw the fan $\Sigma$ superimposed on the amoeba $\cA$ of the tailored hypersurface $H^1,$ as illustrated in Figure \ref{fig:example-pc}.
  \begin{figure}[h]
    \begin{center}
      \includegraphics[width=6cm]{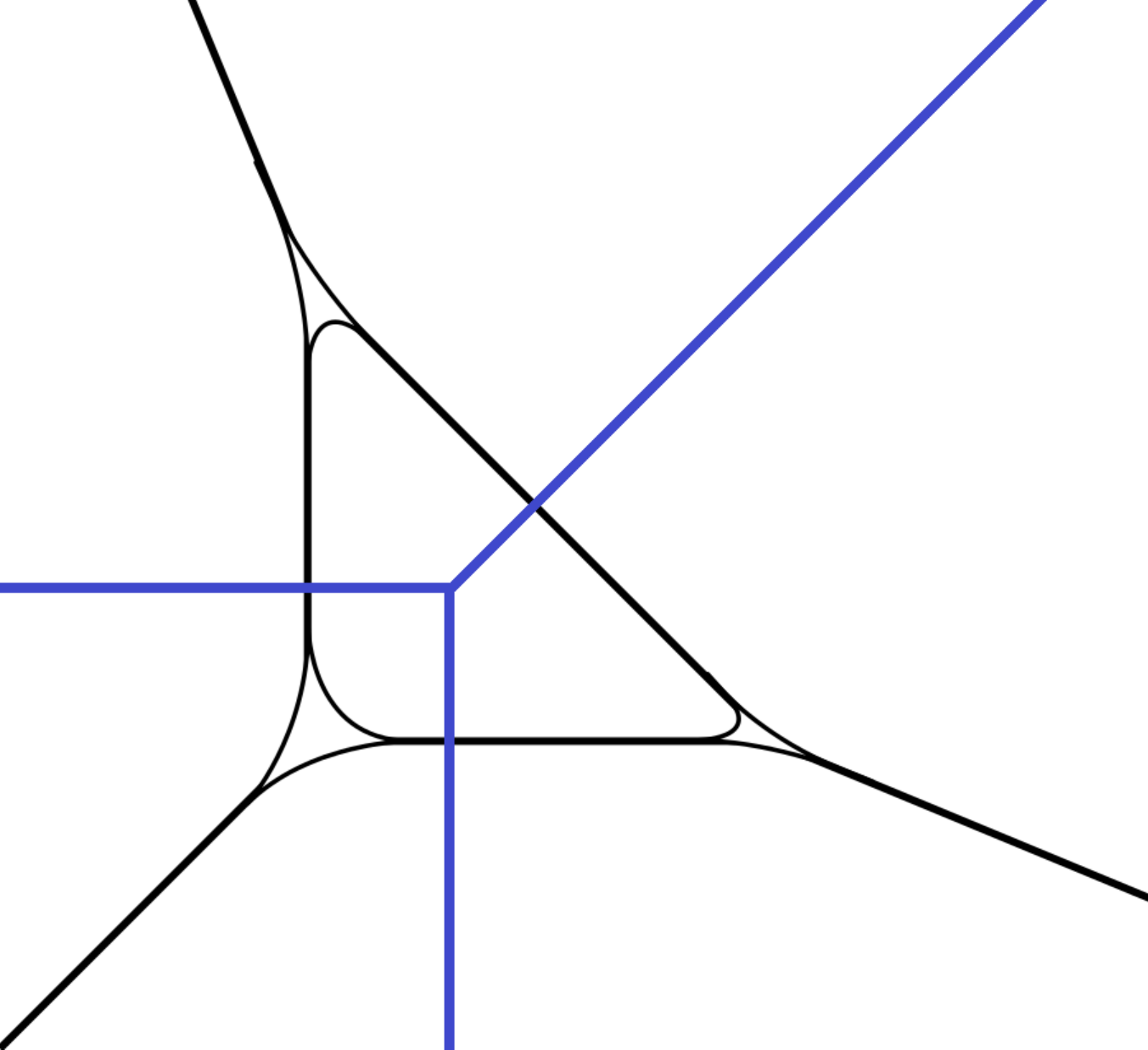}
    \end{center}
    \caption{The fan of $\PP^2$ superimposed on the tailored amoeba of its mirror hypersurface $H = \{x+y+\frac{1}{xy}=0\}.$ Note how this tailored amoeba is precisely ``tropical'' away from the vertices.}
    \label{fig:example-pc}
  \end{figure}
  As mentioned above, the complement of $\cA$ has a distinguished component, dual to the origin of $\Sigma,$ thought of as a vertex in the triangulation $\cT.$ We denote by $\cA_0$ the boundary of this region, and $\Trop(H)_0$ for the corresponding subcomplex of the tropical curve.
  \begin{lemma}[\cite{M}]
    Suppose the coefficients in the Laurent polynomial defining $M$ are real positive, except for the constant term which is real negative. Then the set
    \[
    H_+:=H\cap (\RR_{>0})^n
    \]
    of positive real points of $H$ map to $\cA_0$ under the $\Log$ map.
    \label{lem:posreal}
  \end{lemma}
  Now, one computes that for each $k$-face $F$ in the polytope $\Trop(H)_0$, there corresponds a single Morse-Bott critical point (indicated by the intersection of $F$ with its dual cone in $\Sigma$) with critical locus an $(n-k)$-torus $T_F$. The Lagrangian skeleton of $H$ can then be described as a union, over cones $\sigma$ in $\Sigma$ (dual to faces $F_\sigma\subset \Trop(H)_0$),
  \[
  \LL_H = \bigcup_{0\neq \sigma \in \Sigma} T_{F_\sigma}\times \Log^{-1}(\sigma\cap \cA_0),
  \]
  of the downward Liouville flows of these tori. This computation completes the proof of Theorem \ref{thm:gs-skel}.

  \begin{remark}\label{rem:skel-sphere}
  Observe that the positive real locus $H_+$ is a subset of $\LL_H,$ presented as a union of $(n-1)$-simplices corresponding to the top-dimensional cones $\sigma$ of $\Sigma.$ If the fan $\Sigma$ is complete, then $H_+$ will be a sphere $S$ (or a disjoint union of spheres if the fan $\Sigma$ is stacky).
  \end{remark}

  \subsubsection{Examples}
Now, to simplify notation, we choose a basis $N\cong \ZZ^n,$ and we begin to consider the cases of particular interest to us. The fundamental example is the following:
\begin{example}\label{ex:skel-an}
  Suppose that $\Sigma = \Sigma_{\AA^n}$ is the standard fan of affine space, with rays spanned by basis vectors $e_1,\ldots,e_n$ of $\ZZ^n.$ If $n=1,$ then $\LL_{\Sigma_{\AA^1}}$ is a ``circle with spoke attached'': the union of $S^1$ with a single conormal direction at $1\in S^1,$ as pictured on the left in Figure~\ref{fig:an-skels}. In general, $\LL_{\Sigma_{\AA^n}}=(\LL_{\Sigma_{\AA^1}})^n$; the case $n=2$ is pictured on the right in Figure~\ref{fig:an-skels}.

\begin{figure}[h]
  \begin{center}
    \includegraphics[scale=0.5]{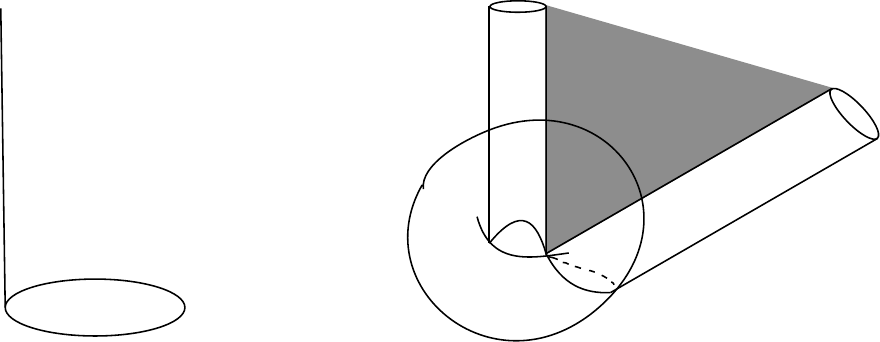}
  \end{center}
  \caption{The skeleta $\LL_{\Sigma_{\AA^n}}$ of the Liouville sectors mirror to $\AA^n$ for $n=1,2.$}
  \label{fig:an-skels}
\end{figure}
\end{example}

The sector $((\CC^\times)^n, 1+x_1+\cdots+x_n)$ described in \ref{ex:skel-an} is mirror to affine space $\CC^n,$ but we will be more interested in the boundary of this mirror symmetry.
\begin{example}\label{ex:skel-pants}
  The boundary $\partial \LL_{\Sigma_{\AA^n}}$ of the skeleton described above is a skeleton for the pants
  \begin{equation*}
		\Pc = \{x_1+\cdots + x_n = 1\}\subset (\CC^\times)^n,
  \end{equation*}
  which is mirror to the toric boundary $\partial \CC^n$ of affine space. This toric boundary is a union of closed pieces $\overline{\cO}_\sigma,$ where we write $\cO_\sigma$ for the toric orbit corresponding to a nonzero cone $\sigma$ in $\Sigma.$ The closed piece $\overline{\cO}_\sigma$ is itself a toric variety, with quotient fan $\Sigma/\sigma.$ Hence in this case it is an affine space, with skeleton as described in Example \ref{ex:skel-an}.

  The skeleton of $\Pc$ can thus be described as follows: consider an $(n-1)$-simplex $\Delta$ (which one can imagine as the ``boundary'' of the Newton polytope for $\Pc,$ where any face containing zero is considered part of the interior). The skeleton of $\Pc$ can be understood topologically as a union of copies of $T^k\times F$ for each nonempty $(n-1-k)$-dimensional subsimplex $F\subset \Delta$ (including $\Delta$ itself), attached according to the face poset of $\Delta.$ From the perspective of each torus, the attachment is through conormal tori as described in Example \ref{ex:skel-an}. 
\end{example}

\begin{example}\label{ex:skel-vn}
	Now consider the very affine Milnor fiber $\Vc = \{W = 1\}\cap (\CC^\times)^n$ of the invertible polynomial $W$, which may be presented as an unramified $G$-cover $\Vc\to \Pc$ of the pants described in the previous example. Its skeleton, which is a $G$-cover of the skeleton $\partial\LL_{\Sigma_{\AA^n}}$ of the pants $\Pc,$ may be described in a similar way: let $\Sigma_{\AA^n/G^\vee}$ be the fan generated by the rays corresponding to monomials in the function $W$; as the notation suggests, this is indeed a stacky fan for the quotient stack $\AA^n/G^\vee.$ Then $\partial\LL_{\Sigma_{\AA^n/G^\vee}}$ is a skeleton for $\Vc.$
\end{example}

\begin{example}\label{ex:pn-bdryskel}
  The other main example of interest to us will be the mirror to the toric boundary of $\PP^n,$ which is the hypersurface $H=\{x_1+\cdots x_n + \frac{1}{x_1\cdots x_n}=n\}.$ Once again, we can use the fact that $\partial \PP^n$ is a union of toric orbit closures to write the skeleton $\LL_H$ of $H$ as glued together from skeleta $\LL_{\Sigma_{\PP^k}}$ of the mirrors to $\PP^k$ for $k<n.$

  The resulting skeleton for $H$ can be described just as in the second paragraph of Example \ref{ex:skel-pants}, except that instead of starting with an $(n-1)$-simplex $\Delta,$ we start with $n+1$ copies of $\Delta,$ glued into the boundary of an $n$-simplex (this time, literally the boundary of the Newton polytope for $H$), which we can understand as a triangulation of the sphere $S$ described in Remark \ref{rem:skel-sphere}. The tori and subtori in $\LL_H$ are attached according to the combinatorics of this complex, just as in Example \ref{ex:skel-pants}.
\end{example}

\begin{example}\label{ex:png-bdryskel}
  As in Example \ref{ex:skel-vn}, we now consider the analogue of Example \ref{ex:pn-bdryskel} for a general invertible polynomial $W$. Consider the intersection
	\begin{equation}\label{eq:intersection}
		H_G := \{W = n\} \cap \left\{\prod_j x_j^{\sum_k a_{kj}}=1\right\},
	\end{equation}
	which agrees with the space $H$ defined in the previous example when $(a_{ij}) = (\delta_{ij})$ is the identity matrix. By construction, $H_G$ may be presented as an unramified $G$-cover $H_G\to H,$ and as a result the skeleton $\LL_{H_G}$ is a $G$-cover of the skeleton $\LL_H$ described in the previous example. Note that in general, the space $H_G,$ and therefore also the skeleton $\LL_{H_G}$ may have multiple components, due to the fact that the second component in the intersection \ref{eq:intersection} is in general a disjoint union of copies of $(\CC^\times)^{n-1}.$
\end{example}

\section{Geometry of the Milnor fiber}\label{sec:geom}
In \S \ref{sec:main}, we will compute the wrapped Fukaya category of the Milnor fiber $V = \{W=1\}\subset \CC^n.$ As preparation, we describe in this section the symplectic geometry of the space $V.$ Ultimately, our goal is to decompose $V$ into a union of several sectors such that mirror symmetry equivalences for these local pieces, and their relations to each other, are already understood explicitly.

In order to simplify the exposition in this section, we begin by descending along the
ramified $\G$-cover $\rho:\CC^n\to \CC^n$ defined at \eqref{eq:in-cover} in \S\ref{sec:intro}, 
so that we may reduce to the study of the simpler Milnor fiber
\[P:=\{x_1+\cdots x_n=n\}.\]
The map $\rho$ restricts to an unramified $\G$-cover $(\CC^\times)^n\to (\CC^\times)^n$ and an unramified cover of the {\em $(n-1)$-dimensional pants} $\Pc,$
\[\Vc\to \Pc:=\{x_1+\cdots+x_n=n\}\subset (\CC^\times).\] 
For most of \S \ref{sec:geom}, we will content ourselves with studying the spaces $P,\Pc$ instead of $V,\Vc$ ---  i.e., we restrict ourselves to the case $W= x_1+\cdots+x_n.$ In \S \ref{sec:covers}, we will explain how to return to the general case.

\subsection{Sectors from a fibration}
We will study the manifolds $P,\Pc$ through the map
\begin{equation}
  \xymatrix{\CC^n\ar[r]&\CC,&(x_1\,\ldots,x_n)\ar@{|->}[r]&x_1\cdots x_n}.
  \label{eq:x1xn}
\end{equation}
The restriction of \eqref{eq:x1xn} to $P$ 
we will denote by $f$, and the further restriction to $\Pc$ 
we will denote by $\fc,$ or possibly also by $f$ in cases where this will not cause confusion. (In \S \ref{sec:covers}, we will study $V,\Vc$ using the restrictions of the map obtained from \eqref{eq:x1xn} by precomposition with the $G$-cover $\rho.$ We will also denote these restricted maps by $f,\fc$.)
We will need the following facts about this map.
\begin{lemma}\label{lem:ffacts}
  \begin{enumerate}
    \item The map $\fc$ has unique critical value $\{1\}\subset \CC^\times,$ corresponding to a single nondegenerate critical point at $(1,\ldots,1).$
    \item The general fiber of $f$ or $\fc$ is a hypersurface in $(\CC^\times)^{n-1}=\{x_1\cdots x_n = c\}.$ In coordinates $x_1,\ldots,x_{n-1}$ on this $(\CC^\times)^{n-1},$ this hypersurface is defined by the equation $x_1+\cdots + x_{n-1} + \frac{c}{x_1\ldots x_{n-1}}=n.$
  \end{enumerate}
\end{lemma}

\begin{proof}
  Part (3) is clear. For parts (1) and (2), we parametrize $P$ by $x_1,\ldots,x_{n-1},$ so that the map $f$ becomes
  \begin{equation}
    (x_1,\ldots,x_{n-1})\mapsto
    \left(\prod_{j=1}^{n-1} x_j\right)\left(n-\sum_{j=1}^{n-1} x_j\right),
    \label{eq:fparam}
  \end{equation}
  and hence the $i$th derivative of this map is given by
  \begin{equation}
    \frac{\partial f}{\partial x_i} =
    (\prod_{j\neq i}x_j)\left(n-x_i-\sum_{j=1}^{n-1} x_j\right).
    \label{eq:ithder}
  \end{equation}
These derivatives all vanish simultaneously only if all $x_i=0,$ or if all $x_i = 1,$ and the latter point (unlike the former) has a nondegenerate Hessian. 
\end{proof}


The map $f$ is compatible with a Liouville-sectorial decomposition of the hypersurface $P$:

\begin{lemma}\label{lem:sectorial-p}
	Let $F=f^{-1}(\frac12)$ be a general fiber of $f$, let $I\subset \RR$ be a closed interval, and let $\DD$ be a closed $(n-1)$-disk. Then
	the Weinstein manifold $P$ (resp. $\Pc$) can be presented via a Liouville-sectorial gluing $P = P_L\cup_{P_{\frac12}} P_R$ (resp. $\Pc_L\cup_{P_{\frac12}} P_R$) where $P_{\frac12}$ is equivalent to the sector $T^*I \times F$ and
	$P_R$ is equivalent to the sector $T^*\DD$.
\end{lemma}
\begin{proof}
	Let $\lambda_{\frac12}$ be a Liouville form on the fiber $F$ (obtained for instance by restricting the standard Stein potential $\sum_{i=1}^n |x_i|^2$ from $\CC^n$)
	and let $\lambda_\CC$ (resp. $\lambda_{\CC^\times}$) be a Liouville form on $\CC$ (resp. $\CC^\times)$ whose skeleton is the interval $[0,1]$ (resp. the union of a radius-$r$ circle about 0 and the interval $[r,1],$ where $0<r\ll\frac12$.)
	Above a small ball $B$ about $\frac12\in \CC,$ the manifold $P$ (or $\Pc$) is equivalent to a product $F\times B$, and $\lambda_{\frac12}+f^*(\lambda_{\CC}|_B)$ (or $\lambda_{\frac12}+f^*(\lambda_{\CC^\times}|B)$) can be used to equip this space with Weinstein sectorial structure of $F\times T^*I.$

	We would like to extend this to a Weinstein structure on the whole space $P$ (or $\Pc$), which requires explaining what happens over the critical values of the map $f$. The critical value 1 corresponds to a Lefschetz critical point, which entails a single handle attachment to $F\times T^*I.$ Extending Weinstein structure across such a handle is a standard construction --- see for instance \cite[\S 6]{Giroux-Pardon}. The result will be a Weinstein sector $P_R$ whose potential function has a single critical point at the center of the handle $\DD$ to be attached, and as a result this sector will be equivalent to $T^*\DD.$

	For $\Pc$ we are now done (since the Liouville structure extends without problems to the left-hand sector $\Pc_L$, where $\fc$ is a fibration), but for $P$ we still need to extend the Liouville structure over the critical value at 0. To do this, we add to our Stein potential a term coming from the function $(\epsilon \psi)\sum_{i=1}^n|x_i|^2,$ where $\epsilon\in \RR$ is a constant satisfying $0<\epsilon\ll 1$ and $\psi$ is a bump function which is 1 near the preimage of a ball around $0\in \CC$ and 0 elsewhere. This gives a Weinstein structure near the preimage of 0, and in the region where $\psi$ is nonconstant, assuming we have chosen $\epsilon$ sufficiently small, the contribution of this term to the Weinstein structure is negligible with respect to the other terms $\lambda_{\frac12}+f^*\lambda_\CC,$ so this extends to a Weinstein structure on $P$.
\end{proof}

%
%
%

From Lemma \ref{lem:sectorial-p}, we see that the sector $P_R$ 
represents a single Weinstein disk attachment to the Wein\-stein manifold $P_{\frac12},$ and the only further information we need in order to understand $P_R$ as a subsector of $P$ is a description of how this disk is attached: we need to identify the Lagrangian sphere $S = \partial \DD$ inside the skeleton of $P_{\frac12}.$
%
%
%

The attachment of the disk $\DD$, or in other words the degeneration of the fiber of $f$ over $1\in \CC,$ can be most easily described using tropical geometry.
Recall that the complement of the amoeba $\cA_M$ has a single bounded region, whose boundary is the diffeomorphic image of the real positive locus $M_+$ inside $M$.

\begin{lemma}
  The boundary of the Lagrangian disk $\DD$ attached at $\{1\}$ is the real positive locus $M_+\subset M.$
  \label{lem:disk-attach}
\end{lemma}

\begin{proof}
	As explained above, the Lagrangian disk $\DD$ is the handle attached to the product $B\times F,$ where $B$ is a disk in $\CC$ and $F=f^{-1}(\frac12)$ is a general fiber, by a single Weinstein handle attachment, corresponding to the Leschetz singularity of the function $f:P\to \CC$ over the critical value $1\in \CC.$ The Legendrian sphere in $F$ along which this handle is attached is precisely the vanishing cycle of $F$ corresponding to this Lefschetz singularity; the whole Lagrangian handle $\DD$ is the Lefschetz thimble for this singularity.

	Therefore, we need to find the vanishing cycle for this Lefschetz singularity. Parallel transport to the fiber over 1 collapses the Lagrangian sphere $M_+\cong S^{n-2}\subset F$ to the point $(1,\ldots,1),$ so we conclude that $M_+$ is the vanishing cycle associated to the critical point over 1. This can be seen most clearly from a tropical perspective: the degeneration of the amoeba $\cA_M$ illustrated in Figure \ref{fig:tropdeg1} should be read as a movie depicting the Lefschetz thimble filling in the interior of $M_+.$
\end{proof}

\begin{figure}[h]
  \begin{center}
    \includegraphics[width=5in]{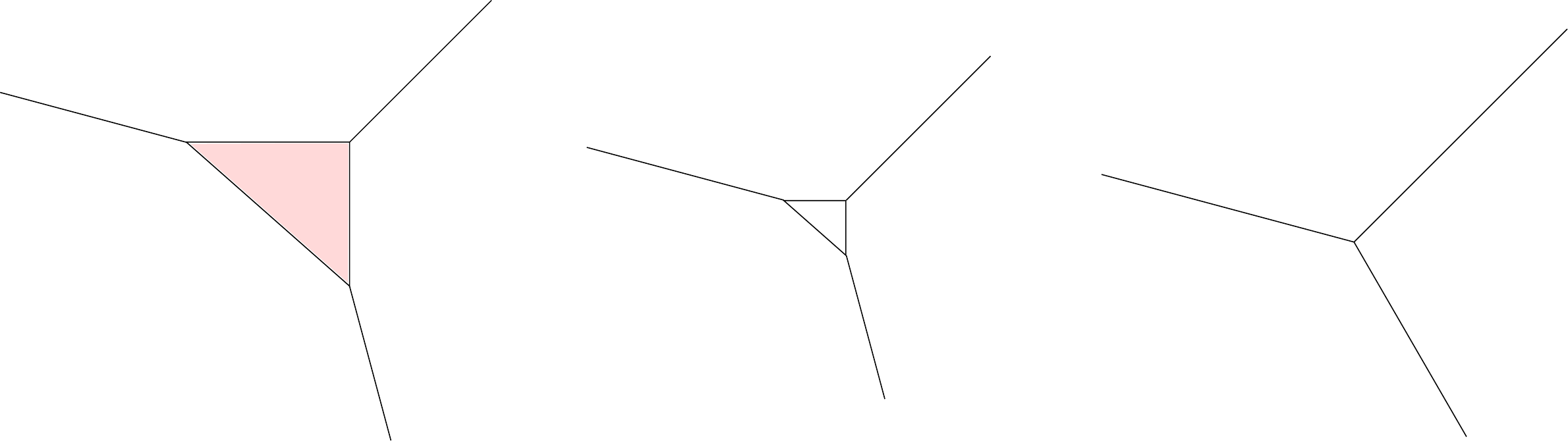}
  \end{center}
  \caption{The degeneration of the tropical hypersurface of $M$ as it approaches the critical value $1$. In the figure on the left, the collapsing region where the disk is to be attached is shaded.}
  \label{fig:tropdeg1}
\end{figure}

In other words, the disk $\DD$ is attached to the skeleton $\LL_{\Pc_{\frac12}}$ of a general fiber
along the sphere $S$ corresponding to the top-dimensional cones in the fan $\Sigma_{\PP},$ as described in Remark \ref{rem:skel-sphere}.

\subsection{The sector around 0}
As with the sector $P_R,$ the left-hand sector $P_L$ can be understood by studying a map $P_L\to \CC_{\Re\leq\frac12}$ with a single critical value. However, unlike in the case of $P_R,$ the critical value of $f|_{P_L}$ does not correspond to a Picard-Lefschetz singularity; in fact, the singularity above $0$ is not isolated, nor is it Morse-Bott;  it is built out of the normal-crossings degenerations we shall discuss in \S \ref{sec:nxings}.

By contrast, the punctured sector $\Pc_L$ is significantly simpler to understand than $P_L$ since the map $\fc|_{\Pc_L}$ has no critical points. Let $P_{\pm\epsilon}$ be the fibers over this map over $\pm\epsilon$ for some $\epsilon\in (0,1),$ and write $\phi_\pm:P_{\epsilon}\simeq P_{-\epsilon}$ for the two identifications of these fibers, given by parallel transport above and below $0,$ respectively. Also write $\cA_2$ for the Liouville sector given by a disk with three stops on the boundary.
\begin{proposition}\label{prop:leftskel-va}
  The sector $\Pc_L$ is obtained from the product sector $P_\epsilon\times \cA_2$ by gluing together two of the ends using the identification $(\phi_-)^{-1}\circ \phi_+$. 
\end{proposition}
\begin{proof}
  The sector $\Pc_L$ has a map $\fc$ to the sector given by an annulus with one stop on its boundary, and this map has no critical points. Hence this sector admits a Liouville form by adding Liouville forms on the base and on the fiber $P_\epsilon.$ If we choose a Liouville structure on the base with skeleton the ``lollipop'' $\lol,$ then the sector $\Pc_L$ will have a skeleton living over the lollipop, which has one $A_2$ singularity with two of its ends glued together; in the fiber, the identifications on these two ends with $P_\epsilon$ will differ by the half-monodromies $\phi_\pm.$
\end{proof}

\begin{remark}\label{rem:totalskel-va}
	We can use the above description of $\Pc_L$ to understand a Lagrangian skeleton for the space $\Pc.$ In fact, $\Pc$ is an $(n-1)$-dimensional pants, and we have already described its skeleton $\LL_{\Pc}$ in Example \ref{ex:skel-pants}. But the relation of this skeleton to $\Pc_L$ is interesting.

  Recall that $\LL_{\Pc}$ has a sectorial cover indexed by the face poset of an $(n-1)$-simplex, where a $k$-face of $P_L$ corresponds to the sector with FLTZ skeleton $\LL_{\Sigma_{\AA^{n-1-k}}}\times T^*\RR^k,$ with the interior of the simplex corresponding to sector $T^*\RR^{n-1},$ which has skeleton a disk $\DD.$ Write $\LLc$ for the complement $\LL_{\Pc}\setminus \DD$ of this disk in the skeleton of $\Pc.$

\begin{figure}[h]
  \begin{center}
    \includegraphics[scale=0.7]{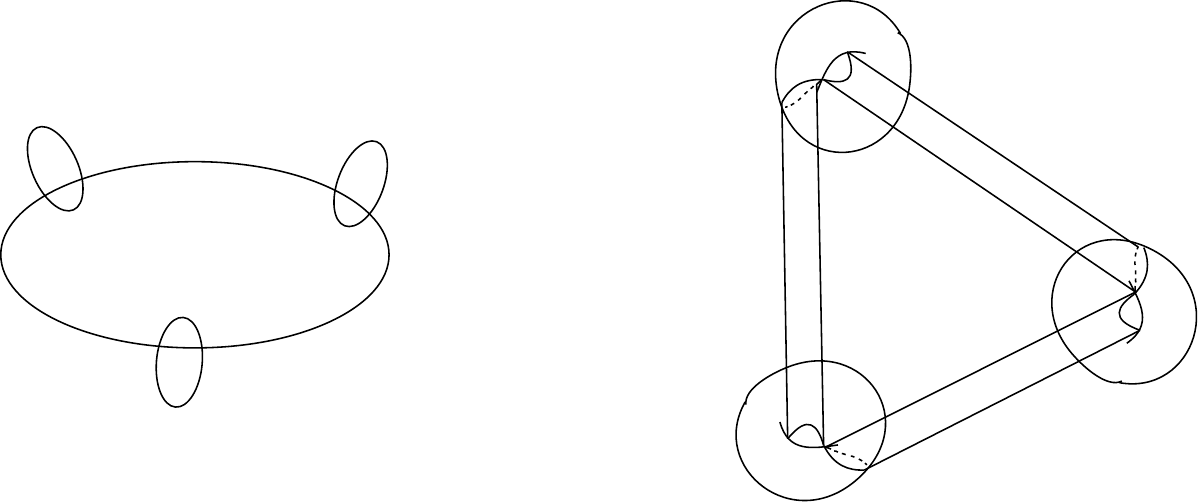}
  \end{center}
  \caption{The Lagrangians $\LL_M$ (left) and $\LLc$ (right) for $n=3.$ Note that the $\RR^k\times T^{n-k-2}$ pieces of $\LL_M$ become copies of $\RR^k\times T^{n-k-1}$ after being swept out by the parallel transport around 0.}
  \label{fig:skel-rotated}
\end{figure}

  Then $\LLc$ is precisely the Lagrangian swept out in $\Pc_L$ by parallel transport of the skeleton of a general fiber $\LL_{\Pc_{\frac12}}$ around 0, as illustrated in Figure \ref{fig:skel-rotated}. This is the Lagrangian skeleton of the Liouville {\em subdomain} $\fc^{-1}(\mathring{\DD})$ above a small punctured disk $\mathring{\DD}$ around 0. The total skeleton $\LL_{\Pc}$ of the space $\Pc$ is obtained from this by the disk attachment indicated by the critical value in $\Pc_R,$ which attaches the interior disk to the skeleton.
\end{remark}

\subsection{Covering spaces}\label{sec:covers}
Thus far we have described the geometry only of the space $P$, and its open locus $\Pc$, rather than the more general Milnor fiber $V$ (and its very affine part $\Vc$). However, 
as we now explain, the case of general $V$ immediately follows from this one.

The pants $\Pc\subset (\CC^\times)^n$ is defined by the Laurent polynomial $x_1+\cdots+x_n-n,$ which has Newton polytope 
\[\Newt(x_1+\cdots+x_n-n) =\Delta_n:=\Conv(0,e_1,\ldots,e_n)\subset \RR^n\] 
the standard simplex. The key fact we use, which played an essential role in the constructions of \cite{GS17}, is that the Newton polytope of the function $W-n$ is the larger simplex 
\[
\Newt(W-n)=\Delta_W:=\Conv\left(0,\sum_j a_{1j},\ldots,\sum_j a_{nj}\right)\subset \RR^n,
\]
where as usual we write
\[W = \sum_{i=1}^n\prod_{j=1}^n x_j^{a_{ij}}.\]

The matrix $A=(a_{ij})$ defines a linear map $\RR^n\to\RR^n$ taking $\Delta_n$ to $\Delta_W,$ hence a map $\rho:(\CC^\times)^n\to(\CC^\times)^n$ which restricts to an unramified $G$-cover $\Vc\to \Pc,$ as described in the Introduction. This extends to a ramified $G$-cover
\[
	\rho:\CC^n\to \CC^n,
\]
restricting to a map $V\to P$ which is a $G$-cover over its image, with no ramification outside the fiber over $0$.

Passing along the cover $\Vc\to\Pc$ (resp. $V\to P$), we can immediately lift our description of the sectorial decomposition of $\Pc$ (resp. $P$) to the space $\Vc$ (resp. $V$).  Consider the fan of cones on the faces of $\Delta_W,$ and write $\Sigma_W$ for its image in the quotient $\RR^n/\RR_\Delta$ by the diagonal copy of $\RR^n.$ Then the $G$-cover of the sectorial decompositions for $\Pc,P$ described above are as follows.
\begin{proposition}
	The space $V$ (resp. $\Vc$) can be presented as a sectorial gluing $V= V_L\cup_{V_{\frac12}} V_R$ (resp. $\Vc = \Vc_L\cup_{V_{\frac12}}V_R$).
  \begin{itemize}
  	\item The central sector $V_{\frac12}$ is a product $F\times T^*I$, where $I$ is an interval and $F$ is a fiber of $f,$ admitting a Lagrangian skeleton given by the boundary FLTZ Lagrangian $\partial\LL_{\Sigma_W}.$ 
	\item The right-hand sector $V_R$ is a disjoint union $\bigsqcup_{i=1}^{|G|}T^*\DD$ of $|G|$ copies of the cotangent bundle of a closed disk, attached to the $|G|$ lifts in $\Vc_{\frac12}$ of the real positive sphere $S$ in $\Pc_{\frac12}.$
	\item The left-hand sector $\Vc_L$ can be obtained from $\Vc_{\frac12}$ by taking a product with the $\cA_2$ sector and gluing two ends together by the monodromy isomorphism.
  \end{itemize}
  \label{prop:G-va-sectcov}
\end{proposition}


\section{Fukaya categories from deformation theory}\label{sec:section4}
In the previous section, we gave detailed descriptions of the sectors $\Vc_L,V_{\frac12},$ and $V_R$, which can (and will in \S \ref{sec:main}) be used to compute the wrapped Fukaya categories of these sectors and ultimately their gluing, the wrapped Fukaya category $\cW(\Vc)$ of $\Vc.$
However, we saw that the partially compactified sector $V_L$ was more complicated than $\Vc_L$ in general.

Ideally, we would like to compute $\cW(V_L)$ using our knowledge of $\cW(\Vc_L)$ and a small amount of extra data. In \S \ref{sec:nxings}, we will explain how this was accomplished for $P_L,\Pc_L$ in \cite{Nad-cnw}; afterward, we will reconceptualize this argument as a general principle about Fukaya categories, which we will summarize as Conjecture \ref{conj:main}.

\subsection{Mirrors to normal crossings}\label{sec:nxings}
Following the constructions of \cite{Nad-cnw}, we now discuss a proof of homological mirror symmetry for the $(n-2)$-dimensional pair of pants, which is now understood as living on the B-side.
\begin{definition}
  We denote by $\nsecc_n$ the Liouville sector corresponding to the Landau-Ginzburg model $(\CC^n, x_1,\ldots,x_n).$
\end{definition}
The general fiber $F_1$ of the function $x_1\cdots x_n$ is a complex $(n-1)$-torus, which degenerates to $\{x_1\cdots x_n=0\}$ over the unique critical value 0 of this map. 

\[T = \{x_1\cdots x_n =1, |x_1|=\cdots = |x_n|\}\] 
be the unit torus in the fiber over 1.
As discussed in \cite{Nad-cnw}, the skeleton $\LL_{\nsecc_n}$ of this sector is the parallel transport of $T$ over the real half-line $[0,\infty),$ which collapses $T$ to a point in the fiber over 0. In other words, $\LL_{\nsecc_n}$ is the cone over the compact $(n-1)$-torus $T$.

The wrapped Fukaya category $\cW(\nsecc_n)$ of this sector is computed in \cite{Nad-cnw} by studying the spherical ``cap'' functor
\[
	\wmsh(\LL_{\nsecc_n})\to \wmsh(T)
\]
and realizing this functor as mirror to the pushforward along the inclusion of a linear hypersurface in $F_1^\vee := (\CC^\times)^{n-1}.$ 


Before we explain this calculation, we will replace the sector $\nsecc_n$ by a related Liouville sector.
\begin{definition}\label{def:nsec}
	We denote by $\nsec_n$ the cornered Liouville sector corresponding to the completed LG triple $(\widehat{\CC^n},x_1\cdots x_n, x_1+\cdots + x_n)$ as in Definition \ref{defn:cutoff-triple}.
\end{definition}

Observe that if we restrict the polynomial $x_1 + \cdots x_n$ to the fiber $F_1 = \{x_1\cdots x_n = 1\}$, we obtain the function $x_1+\cdots + x_{n-1}+\frac{1}{x_1\cdots x_n}$ on $F_1\cong (\CC^\times)$; as we have seen, this is the Hori-Vafa superpotential of the mirror to $\PP^{n-1}.$ As we shall see, passing from $\nsecc_n$ to $\nsec_n$ has the effect on the mirror of replacing the linear hypersurface $\Pc_{n-2}\subset F_1^\vee\cong (\CC^\times)^{n-1}$ with its compactification in $\PP^{n-1}\supset (\CC^\times)^{n-1}.$


The definition of completed LG triple is set up to ensure that the skeleton $\LL_{\nsec_n}$ is straightforward to compute from knowledge of the sectors $(\CC^n, x_1\ldots,x_n)$ and $(F_1, x_1+\cdots + x_n)$: it is given by parallel transport, over a ray emanating from $0$, of the skeleton $\LL_{\Sigma_{\PP^{n-1}}}$ of the LG model $(F_1,x_1+\cdots + x_n).$ As described in Example \ref{ex:orlov-triple}, the boundary sector of the cornered LG triple $\nsec_n$ has its own sectorial decomposition: 
\[
	\partial \nsec_n = \partial^h\nsec_n \cup_{\partial^2\nsec_n} \partial^v\nsec_n,
\]
where the {\em horizontal boundary} $\partial^h\nsec_n$ is the LG Liouvile sector $(F_1,x_1+\cdots +x_n)$ with skeleton $\LL_{\Sigma_{\PP^{n-1}}}$; the {\em vertical boundary}, obtained by equipping a subdomain of $\{x_1+\cdots + x_n = n\}$ (on which $x_1\cdots x_n$ has no critical values outside 0) with the potential $x_1\cdots x_n$, has skeleton given by taking the parallel transport of $\partial \LL_{\Sigma_{\PP^{n-1}}}$ to $0$ in the $x_1\cdots x_n$-plane; and the {\em corner} $\partial^2 \nsec_n$ is their intersection, which has skeleton $\partial \LL_{\Sigma_{\PP^{n-1}}}.$


\begin{example}
  When $n=2,$ the Lagrangian $\LL_{\Sigma_{\PP^1}}$ is the union of the circle $T$ with the conormal rays to $1\in T.$ Hence the total skeleton $\LL_{\nsec_2}$ is obtained from the degenerating circle $\LL_{\nsecc_2}$ by attaching a half-plane.
\end{example}
\begin{remark}
	In the previous example, the attached half-plane does not contribute to the calculation of the category $\wmsh(\LL_{\nsec_2}),$ since it is not affected by the degeneration at 0. (This is mirror to the fact that a general hyperplane in $\PP^1$ does not intersect its boundary. In fact, the sector $\nsec_n$ was introduced in \cite{Nad-cnw} precisely to prove the second assertion of Lemma \ref{lem:nad-main} below.) In higher dimensions, the corner skeleton
	contains tori which will degenerate over 0 and hence affect the computation of $\cW(\nsec_n).$ 
	Nevertheless, those components of the corner skeleton which do not contribute to $\cW(\nsec_n)$ still play a role in Lemma \ref{lem:nad-main}.
\end{remark}

\begin{remark}
	The vertical boundary sector $\partial^v\nsec_n$ is easily seen to be equivalent to the sector $P_L$ described in Lemma \ref{lem:sectorial-p}: by definition, each of these sectors is obtained by beginning with a fiber of $x_1+\cdots + x_n$; passing to a region which does not contain the extraneous critical value of $x_1\cdots x_n$; and then adding a stop given by a fiber of $x_1\cdots x_n$.
\end{remark}

The space $\CC^n$ --- and ultimately also the sectors constructed from it --- is equipped with a polarization coming from its presentation as a cotangent bundle $T^*\RR^n.$ Concretely, this means that exact Lagrangians in $\CC^n,$ such as the Lagrangian skeleton $\LL_{\nsec_n}$, may be lifted to conic Lagrangians in $T^*(\RR^{n+1}),$ where microlocal sheaves are defined, and this is how the computations in \cite{Nad-cnw} are accomplished. Thanks to Theorems \ref{thm:global-mush} and \ref{thm:gps-maincomparison}, we can now understand those calculations 
%
in terms of Fukaya categories of sectors rather than microlocal sheaves on the skeleta of those sectors, and we will now use those language. The first of these results is the computation of the monodromy automorphism on the Fukaya category of the horizontal boundary sector, $\cW(\partial^h\nsec_n).$ 

\begin{proposition}[{\cite[Corollary 4.24]{Nad-cnw}}]
  The mirror symmetry equivalence
  \[
    \cW(\partial^h\nsec_n) \cong \Coh(\PP^{n-1})\]
		identifies the clockwise monodromy automorphism $\mu^{-1}$ with the functor $-\otimes\cO_{\PP^{n-1}}(-1)$ of tensoring with the invertible sheaf $\cO_{\PP^{n-1}}(-1).$ 
\end{proposition}

The calculation from \cite{Nad-cnw} now proceeds via the following lemma:
\begin{lemma}[\cite{Nad-cnw}]
  \begin{enumerate}
    \item The horizontal cap functor
  \[
  \cap^h:\cW(\nsec_n)\to\cW(\partial^h\nsec_n)
  \]
  is spherical and monadic. Under the identification $\cW(\partial^h\nsec_n)\cong\Coh(\PP^{n-1}),$ this monad is therefore given by tensoring with the cone of a morphism
  \[
	s:\cO_{\PP^{n-1}}(-1)\to \cO_{\PP^{n-1}}.
  \]
\item The morphism $s$ is {\em generic}: there exist coordinates $z_i$ on $\PP^{n-1}$ in which $s$ is given by the section $z_1+\cdots+z_n\in\Gamma(\PP^{n-1},\cO(1)).$
  \end{enumerate}
  \label{lem:nad-main}
\end{lemma}

The above lemma establishes that the monad associated to the spherical ``horizontal cap'' functor $\cap^h$ is equivalent to the monad for the pushforward functor
\[
	i_*:\Coh(\PP^{n-2})\to \Coh(\PP^{n-1})
	\]
	for the inclusion of a generic hyperplane in $\PP^{n-1}.$ This leads directly to the main result of \cite{Nad-cnw}:

\begin{corollary}[{\cite[Theorem 1.5]{Nad-cnw}}]
  There is a commutative diagram with horizontal equivalences
  \begin{equation} \label{eq:hface-1}
  \xymatrix{
  \cW(\nsec_n)\ar[r]^-\sim\ar[d]&\Coh(H)\ar[d]\\
  \cW(\partial^h\nsec_n)\ar[r]^-\sim&\Coh(\PP^{n-1}),
  }
\end{equation}
where 
$H\cong \PP^{n-2}$ is the hypersurface 
\[
	H:=\{z_1+\cdots +z_n=0\}\subset \PP^{n-1},
\] and
the right vertical map is the pushforward along the inclusion 
$
i:H\hookrightarrow\PP^{n-1}.
$
\end{corollary}

These equivalences form one face of a commutative cube, but we will be more interested in the other face, which can be obtained from the diagram \eqref{eq:hface-1} by restricting to the vertical boundary on the A-side, and restricting to the toric boundary on the B-side:

\begin{corollary}\label{cor:nsec-bdry}
  There is a commutative diagram with horizontal equivalences
  \begin{equation} \label{eq:hvface-1}
  \xymatrix{
  \cW(\partial^v\nsec_n)\ar[r]^-\sim\ar[d]&\Coh(H\cap \partial\PP^{n-1})\ar[d]\\
  \cW(\partial^{2}\nsec_n)\ar[r]^-\sim&\Coh(\partial\PP^{n-1}).
  }
\end{equation}
\end{corollary}

\begin{remark}
	Instead of taking the diagram \eqref{eq:hvface-1} as obtained from \eqref{eq:hface-1} by restriction, one could also obtain \eqref{eq:hvface-1} by gluing together lower-dimensional copies of the equivalence from \eqref{eq:hface-1}, by extending the sectorial cover described in Example \ref{ex:pn-bdryskel} to a sectorial cover $\nsec_n.$
\end{remark}

\subsection{Matrix factorizations}\label{sec:defthry1}
We will now reinterpret the categorical computation described above, so that we can understand the category $\cW(\partial^v \nsec_n)$ as a matrix factorization category. 

Traditionally, the matrix factorization category $\MF(X,f)$ is defined from the input data of a scheme $X$ and a global function $f\in \cO(X).$ As explained in \cite[Appendix]{Teleman-morsebott} and \cite[\S 5]{Preygel}, this category can be understood as a deformation of the 2-periodicized category $\Coh(X)_{\ZZ/2}$ by the Hochschild class $\beta f\in HH^2(\Coh(X)_{\ZZ/2})$, where $\beta$ is the 2-periodicity element. This construction can be performed for a more general category $\cC$ with an element $f\in HH^0(\cC)$, and indeed, we would like to discuss this construction in general, without necessarily assuming the category $\cC$ is a category of coherent sheaves (although ultimately the categories considered below will be of this form). We therefore make the following definition:

\begin{definition}\label{def:mfax-gen}
	Let $\cC$ be a category and $f\in HH^0(\cC)$ specifying deformation class $\beta f\in HH^2(\cC_{\ZZ/2}).$ We write $\MF(\cC,f)$ for the 2-periodic category obtained from this deformation: namely, $\MF(\cC,f)$ has objects given by 2-periodic complexes $c_0\rightleftarrows c_1$ whose differentials square to the respective images of $f$ in $\End_\cC(c_i),$ and morphisms given by maps of complexes.

	When $\cC=\Coh(X)$ and $f\in \cO(X)$, we will abbreviate this category as $\MF(X,f).$
\end{definition}


If $X$ is smooth and $f\in \cO(X),$
then the category $\MF(\Coh(X),f)$ described above is equivalent to the traditional category
of matrix factorizations of $f$ on $X$, justifying our notation.
However, even if $X$ is not smooth, we can nevertheless relate $\MF(\Coh(X),f)$ to a traditional matrix factorization category, using Orlov's equivalence from \cite{Or-eq}, which we can phrase as follows:
\begin{lemma}[\cite{Or-eq}]\label{lem:orlov-1}
	Let $X$ be a hypersurface in a smooth stack $Y$ cut out by a function $f\in\cO(Y).$ Suppose moreover that $W^\vee\in\cO(X)$ is the restriction to $X$ of a function $g$ on $Y$ with no nonzero critical values. Then the category $\MF(\Coh(X),W^\vee)$ defined in Definition~\ref{def:mfax-gen} is equivalent to the category $\MF(\CC_t\times Y,g+tf)$ of matrix factorizations on $Y$ for the function $g+tf.$
  \label{lem:tradmf}
\end{lemma}

\begin{example}
  Let $X=\partial\CC^n\subset \CC^n = Y.$ Then the above lemma gives us an equivalence
  \begin{equation}\label{eq:mainex-def}
		\MF(\Coh(\partial\CC^n/\Gv),W^\vee)\cong \MF^\Gv(\CC_t\times \CC^n, W^\vee + tx_1\cdots x_n).
\end{equation}
Note that the undeformed category $\Coh^\Gv(\partial \CC^n)_{\ZZ/2}=\MF^\Gv(\Coh(\partial \CC^n),0)$ is equivalent to the matrix factorization category $\MF^\Gv(\CC_t\times \CC^n,tx_1\cdots x_n).$ The category \eqref{eq:mainex-def} is related to this one as a deformation by $W^\vee,$ and the main result of this paper will be to see that deformation in symplectic geometry.
\end{example}

Lemma \ref{lem:orlov-1} remains true in a twisted form,
when the function $f$ defining $X$ exists only locally.
In the case where $g=0,$ this reads as follows:

\begin{lemma}
	Let $X$ be a hypersurface in a smooth stack $Y$ cut out by a section $s\in\Gamma(Y,\cL^{-1})$ for the inverse of some line bundle $\cL$ on $Y$. Write $\widetilde{s}\in\cO(\Tot(\cL))$ for the function on the total space of $\cL$ obtained by extending $s$ to a function linear on fibers of $\cL$. Then there is an equivalence of 2-periodic dg-categories
	\begin{equation}\label{eq:twisted-orlov}
	\Coh(X)_{\ZZ/2}\cong \MF(\Coh(\Tot(\cL)),\widetilde{s})
\end{equation}
between the 2-periodicized category $\Coh(X)_{\ZZ/2}$ of coherent sheaves on $X$ and the category of matrix factorizations for $\widetilde{s}$ on $\Tot(\cL).$
  \label{lem:knor-gen}
\end{lemma}
\begin{proof}
	The category of matrix factorizations is local in the Zariski topology. We may therefore compute the right-hand side of \eqref{eq:twisted-orlov} on the preimages of an open cover of $Y$ where the section is given by an actual function. Locally on this cover, we are then reduced to the situation of Lemma \ref{lem:orlov-1}.
\end{proof}

We can apply Lemma \ref{lem:knor-gen} to understand the categories discussed in the previous section as deformations.

\begin{example}\label{ex:cohasdef}
	Let $Y=\PP^{n-1}_\Gv,$ and $W^\vee\in \Gamma(\cO_{\PP^{n-1}_\Gv}(1))$ a generic section.
  Then there is an equivalence
  \begin{equation}
    \Coh(\{W^\vee = 0\})_{\ZZ/2}
    \cong \MF(\Tot(\cO_{\PP^{n-1}}(-1)),\widetilde{W^\vee})
    \label{eq:eqvt-defex1}
  \end{equation}
	between the 2-periodicized category of coherent sheaves on the hypersurface in $\PP^{n-1}_\Gv$ defined by $W^\vee$ and the matrix factorization category of the fiberwise-linear extension of $W^\vee$ to a function on the total space of $\cO_{\PP^{n-1}_\Gv}(-1)$.
%
\end{example}

\subsection{Deforming the Fukaya category}\label{subsec:conj}
We now formalize the main calculation of \cite{Nad-cnw} into a general procedure for computing the Fukaya-Seidel category of an LG model with a single critical value. We state this procedure as Conjecture \ref{conj:main} although, as we will explain, it is expected to hold in general and easy to prove in several of the cases of interest to us.

	Let $(X,f:X\to \CC)$ be a Landau-Ginzburg Weinstein sector where $f$ has no critical values outside the origin $0\in\CC,$ and let $F$ be a general fiber of $f.$

	Recall that the cap functor associated to LG model $(X,f)$ is a spherical functor
	\[
		\cap:\cW(X,f)\to \cW(F),
	\]
	with left adjoint $\cup$, whose monad $\cap \cup$ can be presented as the cone on a natural transformation from the clockwise monodromy automorphism to the identity on $\cW(F):$
	\[
		\cap\cup = \Cone(\mu^{-1} \xrightarrow{s} \id_{\cW(F)}).
	\]

	The natural transformation $s$ can be treated as an element of $HH^0(\cW(\Xc,f|_{\Xc}))$: as we shall see in \S\ref{subsec:veryaffine}, the category $\cW(\Xc,f|_{\Xc})$ can be understood as the category whose objects are a pair $(L, \mu L \xrightarrow{\nu}L),$ where $L$ is an element in (an Ind-completion of) $\cW(F)$ and $\nu$ is a $\mu$-twisted endomorphism of $L.$
	Therefore, we may define an element of $HH^0(\cW(\Xc,f|_{\Xc}))$ which acts on an object $(L,\nu)$ by the composition
	\[
		\xymatrix{L\ar[r]^-{s_{\mu L}}& \mu L\ar[r]^-{\nu}&L.}
	\]
	We denote this element of $HH^0(\cW(\Xc,f|_{\Xc}))$ by $\tils.$

	\begin{conjecture}\label{conj:main}
		There is an equivalence
		\[
			\cW(X,f) \cong \MF(\cW(\Xc,f|_{\Xc}), \tils).
		\]
	\end{conjecture}
	As explained in \cite[\S 1.3]{abouzaid-auroux}, the natural transformation $s$ is a count of holomorphic disks living over a disk containing 0 in the base of the LG model. Conjecture \ref{conj:main} would therefore follow immediately from a sufficiently robust theory of deformations of Fukaya categories by holomorphic disks: objects in the category $\cW(\Xc,f|_{\Xc})$ (at least those avoiding a neighborhood of the deleted fiber) ought to give objects of $\cW(X,f)$ (possibly after being equipped with a weak bounding cochain), with the $A_\infty$ structure of the category deformed by the new count of disks, encoded by $s$, passing through the deleted fiber.

	There are technical obstacles to making the discussion of the previous paragraph rigorous, but it is not difficult to establish Conjecture \ref{conj:main} in some generality, as we now explain.
	\begin{lemma}\label{lem:conservative}
		Suppose the cap functor $\cap:\cW(X,f)\to \cW(F)$ is conservative. Then Conjecture \ref{conj:main} holds. More generally, Conjecture \ref{conj:main} holds for the image of the cup functor $\cup:\cW(F)\to\cW(X,f)$: If we write $\cC$ for the image of the cup functor in $\cW(X,f)$ and $\mathring{\cC}$ for the image in $\cW(\Xc,f|_{\Xc})$ of the corresponding cup functor $\cW(F)\to \cW(\Xc,f|_{\Xc}),$ then $\cC\cong \MF(\mathring{\cC},\tils).$
	\end{lemma}
	\begin{proof} If the cap functor is conservative, we may compute the category $\cW(X,f)$ monadically, as the category of $\cap\cup$-algebras in $\Ind(\cW(F))$: 
		by the presentation of the monad as $\cap\cup = \Cone(\mu^{-1}\to \id_{\cW(F)}),$ we see that this category is precisely $\MF(\cW(\Xc,f|_{\Xc}), \tils).$ Equivalently (by a version of Lemma \ref{lem:knor-gen}), as explained in \cite[\S 1.3]{abouzaid-auroux}, this is the category with the same objects as $\cW(F)$ but with Hom between $L$ and $L'$ given by 
		\[\Hom(L,L')=\Cone(\Hom_{\cW(F)}(L,\mu^{-1}(L'))\to \Hom_{\cW(F)}(L,L')).\]

		For the second part of the lemma, we use the fact that the restriction of $\cap$ to the image of $\cup$ will be conservative.
	\end{proof}

	\begin{example}
		Let $X=\CC^n, f(x_1,\ldots,x_n) = x_1\cdots x_n.$ Then $\Xc = (\CC^\times)^n,$ and $\cW(\Xc,f|_{\Xc})\simeq \Coh((\CC^\times)^{n-1}\times \CC)$; under this isomorphism, $\tils\in \cO((\CC^\times)^{n-1}\times \CC) = \CC[z_1^\pm,\ldots,z_{n-1}^\pm,t]$ is shown in \cite{Nad-cnw} to be the function $t(z_1+\cdots + z_{n-1}).$ This calculation underlies the main theorem of \cite{Nad-cnw}, namely the mirror symmetry equivalence 
		\[
			\cW(X,f)\simeq \MF((\CC^\times)^{n-1}\times \CC, t(z_1+\cdots+z_{n-1}))\simeq \Coh\{z\in (\CC^\times)^{n-1}\mid z_1+\cdots + z_{n-1}=0\},
		\]
		where the first equivalence is of the form described in Conjecture \ref{conj:main}.
	\end{example}

	In some sense, the only obstruction to the failure of the hypothesis of Lemma \ref{lem:conservative} is the possibility that the LG model $(X,f)$ will have critical points ``at infinity'' over the zero-fiber --- i.e., that the change in the topology of the fiber over zero is not due to degeneration but rather to parts of the fiber vanishing; such contributions to the topology of $X$ will contribute to the kernel of the cap functor. (This is the phenomenon mentioned in Remark \ref{rem:zero}.) The extreme case of this situation is where the whole fiber vanishes over 0. In this case, the proof Conjecture \ref{conj:main} is trivial:
	\begin{lemma}Suppose that $X=\Xc.$ Then Conjecture \ref{conj:main} holds.
	\end{lemma}
	\begin{proof}
		In this case, the section-counting transformation $s$ is equal to 0, so that $\MF(\cW(\Xc,f|_{\Xc}),\tils)$ is just equivalent to the category $\cW(\Xc,f|_{\Xc}).$ By assumption, this is equivalent to the category $\cW(X,f).$
	\end{proof}

	A complete proof of Conjecture \ref{conj:main} would involve treating these two situations --- the image of the cup functor and the kernel of the cap functor --- on equal footing.

	\subsection{Deformation data for the Milnor fiber}
	We will want to apply Conjecture \ref{conj:main} to the left-hand sector $V_L$ of the Berglund-H\"ubsch Milnor fiber (which may be understood as an LG model with superpotential $f$).
	We will therefore need to gather together the necessary data about this situation: namely, the Fukaya category $\cW(F)$ of the fiber of $f$, together with its monodromy $\mu\in \Aut(\cW(F))$ and the natural transformation $s:\mu^{-1}\to\id_{\cW(F)}$ underlying the monad of the cup-cap adjunction for the sector $V_L.$
	
	In the basic case where $W(x_1,\ldots,x_n)=x_1+\cdots+x_n,$ this is precisely the computation accomplished in \cite{Nad-cnw}. We now recall the result in that case before generalizing it to $V_L$. The following was already stated as Lemma \ref{lem:nad-main} above, but now we reformulate it in light of Conjecture \ref{conj:main}.

	\begin{lemma}[\cite{Nad-cnw}]
		Let $X$ be the subdomain of $\{x_1+\cdots +x_n=n\}$ obtained by restricting to $\{|x_1\cdots x_n|<\epsilon\},$ and let $f:X\to\CC$ be given by $f(x_1,\ldots,x_n)  = x_1\cdots x_n,$ so $(X,f)$ is equivalent to the sector $\partial^h\nsec_n \cong P_L$; $(\Xc, f)$ is equivalent to the sector $\Pc_L$; and a fiber $F$ of $f$ is equivalent to the corner $\partial^2\nsec_n$. Then $\mu^{-1}\in \Aut(\cW(F))\cong \Aut(\Coh(\partial\PP^{n-1}))$ is given by $-\otimes \cO(-1),$ and $s:\cO(-1)\to \cO$ is a generic linear function.
	\end{lemma}
	\begin{corollary}
		There is an equivalence of categories
		\[
			\cW(P_L) \cong \MF(\cW(\Pc_L), \tils)
		\]
	\end{corollary}
	\begin{proof}
		The conservativity of the cap functor in this case implies (by Lemma \ref{lem:conservative}) that the conclusion of Conjecture \ref{conj:main} holds.
	\end{proof}

	Except for conservativity of the cap functor, the above statements all remain true when we generalize from $P_L$ to $V_L.$ We will describe the data $(\cW(\Vc_L),\mu\in\Aut(\cW(F)),s:\mu^{-1}\to 1)$ in the next section, where we will use it to compute the category $\cW(V_L).$

\section{Homological mirror symmetry}\label{sec:main}


Having already presented the spaces $V,\Vc$ as covered by recognizable Liouville sectors, it remains for us only to recall the calculations of the wrapped Fukaya categories of those sectors, and then to glue the resulting categories together. We begin with $\Vc$: although we already understand the Fukaya category $\cW(\Vc)$ from the results of \cite{GS17}, we give here a different presentation as preparation for the calculation of $\cW(V).$

\subsection{The very affine Milnor fiber}
\label{subsec:veryaffine}
The space $\Vc$ is an unramified $\G$-cover of $\Pc,$ and the mirror $\Vc^\vee$ is obtained from the mirror $\Pc^\vee = \partial \CC^n$ by passing to a $\G^\vee$-quotient. In other words:

\begin{proposition}[\cite{GS17}]\label{prop:vc-mir1}
  There is an equivalence of categories
  \begin{equation}\label{eq:vc-mir1}\cW(\Vc)\cong \Coh^{\G^\vee}(\partial \CC^n)
  \end{equation}
  between the wrapped Fukaya category of $\Vc$ and the category of $\G^\vee$-equivariant coherent sheaves on the toric boundary of $\CC^n.$
\end{proposition}

The proof of Proposition \ref{prop:vc-mir1} proceeds by matching the closed cover of the stack $\CC^n/\G^\vee$ by toric orbit closures to a Liouville-sectorial cover of $\Vc.$ But we would like to express the category \eqref{eq:vc-mir1} in terms of a different Liouville-sectorial decomposition of $\Vc,$ namely the cover by left- and right-hand sectors $\Vc_L,\Vc_R$ discussed in \S \ref{sec:geom}. We begin with $\Vc_L.$

\begin{definition}
  Let $p:\Bl_0\CC^n\to \CC^n$ be the blowup of $\CC^n$ at the origin. We write $\tpc$ for the strict transform of the toric boundary $\partial\CC^n$ under this blowup, and 
  \[\pp:=p^{-1}(0)\cong \PP^{n-1} \]for the exceptional divisor.
	Similarly, we write $\PP_\Gv$ for the exceptional divisor of the blowup at $0$ of $\CC^n/G^\vee$.
\end{definition}
\begin{remark}
	The stack $\PP_\Gv$ is a $\Gv$-quotient of the projective space $\PP^{n-1},$ through the induced action of $\Gv$ on the exceptional divisor of the blowup. If $\Gv$ has a nontrivial subgroup $H$ which acts diagonally on $\CC^n,$ this subgroup will act trivially on $\PP^{n-1},$ so that $\PP_{\Gv}$ will be an Artin stack with generic stabilizer $H$. In terms of mirror symmetry, this will manifest itself as the fact that the mirror to $\PP_{\Gv}$ (discussed in Example \ref{ex:png-bdryskel}) will have $|H|$ components.
\end{remark}

Note that $\tpc$ intersects $\pp$ in its toric boundary $\partial \pp.$ 
This boundary divisor plays the role of mirror to the central Liouville sector $\Vc_{\frac12}$ in our decomposition of $\Vc$:

\begin{proposition}\label{prop:mir-mid}
  The wrapped Fukaya category $\cW(\Vc_{\frac12})$ of the Liouville sector $\cW(\Vc_{\frac12})$ is equivalent to the category $\Coh(\ppp_\Gv)$ of coherent sheaves on the toric boundary of the projective stack $\PP^{n-1}_\Gv.$
\end{proposition}
\begin{proof}
	This is a corollary of the results in \cite{GS17}, proved by matching the closed cover of $\ppp_\Gv$ by toric orbit closures to a Liouville-sectorial cover of $\Vc_{\frac12}.$
	Alternatively, one can recall from Example \ref{ex:png-bdryskel} the description of the skeleton $\LL_{\frac12}$ of $V_{\frac12}$ as an unramified $G$-cover of the skeleton $\partial\LL_{\Sigma_{\PP^{n-1}}}$ (coming from the presentation of $V_{\frac12}$ as an unramified $G$-cover of $P_{\frac12}).$ The effect of taking this $G$-cover is mirror to imposing a $\Gv$ quotient on $\Coh(\ppp)\cong \cW(P_{\frac12}).$
\end{proof}

In fact, the category $\cW(\Vc_{\frac12})$ of a general fiber of the map $\fc$ comes equipped with extra structure.

\begin{definition}
	We write $\mu^{-1}\in \Aut(\cW(\Vc_{\frac12}))$ for the clockwise monodromy automorphism of the category $\Vc_{\frac12},$ obtained from parallel transport of the general fiber of $\fc$ around 0.
\end{definition}

The automorphism $\mu^{-1}$ admits a geometric description on the mirror space $\ppp$.

\begin{lemma}
	As an automorphism of the category $\Coh^{\Gv}(\ppp),$ the functor $\mu^{-1}$ is given by tensor product with the line bundle $\cO_{\ppp_\Gv}(-1),$ the restriction to $\ppp$ of the line bundle $\cO_{\PP/\Gv}(-1)$ on the projective stack $\pp_\Gv.$
  \label{lem:mon-aut}
\end{lemma}
\begin{proof}
	The space $\Vc$ is given by an unramified $G$-cover of $\Pc$, which as we have seen is mirror to the quotient projection $\partial \CC^n/G^\vee \gets \partial \CC^n,$ and this cover restricted to an unramified $G$-cover $\Vc_{\frac12}\to \Pc_{\frac12},$ mirror to the $G^\vee$-quotient $\ppp_{G^\vee}\gets \ppp.$ We have seen that $\Pc_{\frac12}$ is the corner $\partial^2\nsec_n$ of a sector whose horizontal boundary $\partial^h \nsec_n$ is the LG model $((\CC^\times)^{n-1}, x_1 + \cdots + x_{n-1} + \frac{1}{x_1\cdots x_{n-1}},$ whose skeleton is the FLTZ Lagrangian $\LL_{\Sigma_{\PP^{n-1}}}\subset T^*T^{n-1}.$

	The calculation of \cite[Corollary 4.24]{Nad-cnw} establishes that the monodromy $\mu^{-1}$ on $\partial^h\nsec_n$ is given by convolution by an object which is mirror to $\cO(-1)$ on $\Coh(\PP^{n-1})$; on the boundary $\partial^2\nsec_n = \Pc_{\frac12},$ this monodromy autoequivalence is mirror to tensoring by $\cO_{\ppp}(-1)$ on $\Coh(\partial\PP^{n-1}).$ Now when we pass to the $G$-covers of the previous paragraph, we find that the mondromy on $V_{\frac12}$ comes from convolution with an object mirror to $\cO_{\ppp_\Gv},$ as desired.
\end{proof}

We can use the monodromy automorphism $\mu^{-1}$ to give a new method for computation of the wrapped Fukaya category $\cW(\Vc_L).$

\begin{proposition}\label{prop:pantscalc}
  There is an equivalence 
  \[\cW(\Vc_L) \cong \Coh^{\Gv}(\tpc)\]
  between the wrapped Fukaya category $\cW(\Vc_L)$ and the category of $\Gv$-equivariant coherent sheaves on the proper transform of $\partial \CC$ under the blowup at 0. 
\end{proposition}
\begin{proof}
  The space $\tpc/\Gv$ is a toric stack,
  so one possible proof proceeds following \cite{GS17} as usual, matching toric orbit closures with Liouville subsectors. But we will give here a different proof, more closely associated to the description of the skeleton $\LL_{\Vc_L}$ given in Proposition \ref{prop:leftskel-va}.

  From the description in Proposition \ref{prop:leftskel-va}, we can see that the category $\cW(\Vc_L)$ is equivalent to the compact objects in the category of pairs
  \[
		\left\{\left(X\in \Ind(\cW(\Vc_{\frac12})), \mu X \to X\right)\right\}
  \]
  of an object $X$ in an Ind-completion%
	\footnote{It is often necessary to pass to Ind-completions while computing a colimit, and then to return to small categories afterward by passing to compact objects. The Ind-completion remains in the final description here because an object of $\cW(\Vc_L)$ will often have infinite-dimensional ``underlying object'' in $\cW(\Vc_{\frac12})$; this is analogous to the fact that coherent sheaves on $\Spec R$ are not in general finite $R$-modules, but possibly infinite-dimensional $R$-modules which have a finiteness condition on their generation as $R$-modules. }
	of the Fukaya category of the nearby fiber and a $\mu$-twisted endomorphism of $X$, where $\mu\in\Aut\left(\cW(\Vc_{\frac12})\right)$ is the counterclockwise monodromy map.
  Since $\cW(\Vc_{\frac12})\cong \Coh(\ppp_\Gv)$ with $\mu$ given by the functor of tensor product $-\otimes \cO(1),$  we thus have an equivalence
  \begin{equation}\label{eq:alg-relspec}
    \cW(\Vc_{L})\cong \left\{\left(\cF\in\Ind\Coh(\ppp_\Gv),\nu:\cF(1)\to\cF\right)\right\}.
\end{equation}

  Now note that $\tpc$ is the total space $\Tot(\cO_{\ppp}(-1))$ of the bundle $\cO(-1)$ on $\ppp,$ so that it can be described equivalently as the relative $\Spec$
  \[
  \tpc = \underline{\Spec}_{\ppp}\Sym_{\cO_{\ppp}}\cO(1),
  \]
  and hence the category $\Coh^{\Gv}(\tpc)$ is equivalent to the category of coherent sheaves $\cF$ in $\Coh(\ppp_\Gv)$ equipped with the additional data of a map $\cF\otimes \cO_{\ppp_\Gv}(1)\to\cF$ describing the action of the generators of this symmetric algebra. This agrees with the description of the category $\cW(\Vc_{L})$ given in \eqref{eq:alg-relspec}.
\end{proof}

Now recall from Example \ref{ex:skel-pants} that the space $\Pc,$ the $(n-1)$-dimensional pants, is mirror to the toric boundary $\partial\CC^n,$ and there is an equivalence of categories
\[
\cW(\Pc)\cong \Coh(\pcn).
\]
Accordingly, the wrapped Fukaya category of the space $\Vc,$ which is an unramified $\G$-cover of $\Pc$ (and is called the ``$\G$-pants'' in \cite{GS17}) admits a presentation as
\begin{equation}\label{eq:vc-usu}
\cW(\Vc)\cong \Coh^{\Gv}(\pcn).
\end{equation}
This is not obviously identical with the presentation of $\Vc$ via the Liouville-sectorial cover which we have been discussing so far:
\begin{lemma}
  The wrapped Fukaya category $\cW(\Vc)$ of $\Vc$ is equivalent to the colimit
  \begin{equation}\label{eq:vc-unu}
    \colim\left( \Coh^{\Gv}(\tpc)\gets\Coh(\ppp_\Gv) \to \Coh^{\Gv}(\{0\}) \right),
\end{equation}
  where the maps are given by pushforwards along the inclusion $\pp_\Gv\hookrightarrow \tpc/\Gv$ and the projection $\pp_\Gv\to \{0\}/\Gv,$ respectively.
  \label{lem:vc-prescat}
\end{lemma}
\begin{proof}
  This colimit presentation corresponds to the Liouville-sectorial decomposition we have been studying in this paper. We have already proven that the categories in \eqref{eq:vc-unu} match the wrapped Fukaya categories of the sectors $\Vc_L,\Vc_{\frac12},$ and $\Vc_R,$ respectively, so we only need to check that the functors induced by the Liouville-sectorial inclusions of $\Vc_{\frac12}\times T^*\RR$ into $\Vc_L$ and $\Vc_R$ are as described in the lemma. 
  (Actually, we we will check agreement on the left adjoints of these functors, which are easier to understand. And 
  for simplicity, we note that it is sufficient to check the case where $\Vc=\Pc,$ since the $\G$-cover applies uniformly to all of the Liouville sectors involved and hence the $\Gv$ equivariance applies uniformly to all the categories involved in \eqref{eq:vc-unu}.) 
  
  Consider first the left cap functor
  \[
  \cW(\Pc_L)\to\cW(\Pc_{\frac12}).
  \]
  Under the description in Proposition \ref{prop:pantscalc}, this functor is given by the map which takes a pair $(X,\nu:\mu X\to X)$ to the object $\Cone(\nu).$ But in the B-side description from that proposition, the cone on $\cF(-1)\to \cF$ is the pullback of $\cF$ under the inclusion of the zero-section
  \[
  \ppp\hookrightarrow\underline{\Spec}_{\ppp}\Sym_{\cO_{\ppp}}\cO(1) 
  \]

  Now consider the right cap functor
  \begin{equation}\label{eq:vectmap}
  \cW(\Pc_R)\to\cW(\Pc_{\frac12}).
\end{equation}
The wrapped category $\cW(\Pc_R)$ of sector $\Pc_R$ is equivalent to the category $\Perf_\CC$ of finite-dimensional vector spaces, and we need to show that \eqref{eq:vectmap}
sends the 1-dimensional vector space to the structure sheaf $\cO_{\ppp}\in \Coh(\ppp)\cong\cW(\Pc_{\frac12}).$

  Recall that $\Pc_R$ describes a disk attachment to $\Pc_{\frac12}$ with boundary sphere $S\subset \LL_{\Pc_{\frac12}}.$ We therefore need to check that the Lagrangian sphere $S,$ equipped with trivial local system, represents the structure sheaf $\cO_{\ppp}$ under the mirror symmetry equivalence 
  \begin{equation}\label{eq:hmsppp}
    \cW(\Pc_{\frac12})\cong \Coh(\ppp).\end{equation}
    The equivalence \eqref{eq:hmsppp} describes $\Coh(\ppp)$ as a colimit of categories of coherent sheaves on the toric orbit closures of $\ppp,$ corresponding to a cover of $\Pc_{\frac12}$ by Liouville sectors. The basic sectors in this cover are of the form $(T^*T^{n-2},\Lambda_{\Sigma_{\PP^{n-2}}}),$ and in each case the Lagrangian mirror to the structure sheaf $\cO_{\PP^{n-2}}$ is the cotangent fiber at $1\in T^{n-2}.$ These cotangent fibers glue together to form the sphere $S$, matching the gluing of structure sheaves $\cO_{\PP^{n-2}}$ into the structure sheaf $\cO_{\PP^{n-1}}$ of $\PP^{n-1}.$
\end{proof}

By comparing the equivalence \eqref{eq:vc-usu} with Lemma~\ref{lem:vc-prescat}, we can deduce a new colimit presentation of the category $\Coh^{\Gv}(\pcn).$ In fact, it is possible to prove this directly, without any reference to mirror symmetry:
\begin{lemma}\label{lem:colim-undef}
There is an equivalence of categories
  \begin{equation}\label{eq:colim-undef0}
    \Phi: \Coh^{\Gv}(\pcn)\xrightarrow{\sim} \colim\left( \Coh^{\Gv}(\tpc)\gets\Coh(\ppp_\Gv) \to \Coh^{\Gv}(\{0\}) \right).
\end{equation}
\end{lemma}
\begin{proof}
  The functor $\Phi$ is induced from the pullback functor
  \[
  p^*:\Coh^{\Gv}(\pcn)\to \Coh^{\Gv}(\tpc).
  \]
  This functor is fully faithful, hence $\Phi$ is also, and we need only to prove that $\Phi$ is essentially surjective. In other words, we need to check that every object of the colimit in \eqref{eq:colim-undef0} can be identified with an object in $\Coh^{\Gv}(\tpc)$ which is pulled back from $\Coh^{\Gv}(\pcn).$ 

  So let $\cF$ be an object of $\Coh^{\Gv}(\pcn).$ By using the map $p^*p_*\cF\to \cF,$ we can reduce to the case that $\cF$ is supported on the exceptional locus $\pp_\Gv.$ 
  But the colimit in \eqref{eq:colim-undef0} identifies any such object with the pullback of some sheaf on $\pcn/\Gv$ supported at at $\{0\}/\Gv,$ as desired.
\end{proof}

\subsection{Deformation theory}\label{sec:defthry2}
We are ready at last to proceed to the calculation of the wrapped Fukaya category $\cW(V)$ of the Milnor fiber $V$.
On the B-side, the geometric fact we will need is the deformation of \eqref{eq:colim-undef0} by the function $W^\vee.$
\begin{lemma}
There is an equivalence of 2-periodic categories
\begin{equation}\label{eq:colim-def}
	\Phi: \MF^{\Gv}(\pcn,W^\vee)\xrightarrow{\sim} \colim\left( \MF^{\Gv}(\tpc,W^\vee)\gets\Coh(\ppp_\Gv)_{\ZZ/2} \to \Coh^{\Gv}(\{0\})_{\ZZ/2} \right).
\end{equation}
\end{lemma}
\begin{proof}
This statement is proved in exactly the same manner as Lemma \ref{lem:colim-undef}: the pullback map $p^*:\Coh^{\Gv}(\pcn)\to \Coh^{\Gv}(\tpn)$ induces a fully-faithful embedding 
\[p^{*,W^\vee}:\MF(\pcn/\Gv,W^\vee)\to\MF(\tpn/\Gv,W^\vee)\]
on matrix factorization categories,
and every object of $\MF(\tpn/\Gv)$ is identified in the colimit with one obtained through this deformed pullback. (Note that the function $W^\vee$ fanishes at $0\in \partial \CC^n,$ hence also on the exceptional divisor $\ppp_\Gv,$ and therefore the middle and right-hand categories in the colimit in \eqref{eq:colim-def} could also be written as $\MF(\ppp_\Gv,W^\vee)$ and $\MF^{\Gv}(\{0\},W^\vee),$ respectively.)
\end{proof}

On the A-side, we will need one final piece of data in order to make contact with the description from the previous lemma. Recall that the presentation of $V_{\frac12}$ as a boundary sector of $V_L$ induces a monad $\cap\cup$ on $\cW(V_{\frac12})$.
\begin{lemma}\label{lem:generic}
	Under the mirror symmetry equivalence described in Proposition \ref{prop:mir-mid}, the monad $\cap \cup$ is the endofunctor of $\Coh(\ppp_\Gv)$ given by tensoring with the cone of $s$, where $s$ is the restriction to $\ppp_\Gv$ of the function $W^\vee\in \Gamma(\pp_\Gv,\cO(1)).$
\end{lemma}
\begin{proof}We have already computed in Lemma \ref{lem:mon-aut} that the clockwise monodromy automorphism $\mu^{-1}$ of $\cW(V_{\frac12})\cong \Coh(\ppp_\Gv)$ is given by tensoring with the line bundle $\cO_{\ppp_{\Gv}}(-1),$ and we saw in Theorem \ref{thm:sylvan} that the cap-cup monad admits a presentation as $\Cone(\mu^{-1}\to \id_{\cW(F)}),$ which in our case is therefore a natural transformation between the functor of tensoring with $\cO_{\ppp_\Gv}(-1)$ and the identity functor (which is tensoring with $\cO_{\ppp_\Gv}$), or in other words (after tensoring with $\cO(1)$), a map
	\[
		\cO_{\ppp_\Gv}\to \cO_{\ppp_\Gv}(1),
	\]
	i.e., a section on $\ppp_\Gv$ of the bundle $\cO_{\ppp_\Gv}(1).$ Moreover, this section is generic: this can be seen as in \cite[Theorem 5.1]{Nad-cnw} (the case of $P_L$, discussed in the previous section) by observing that the monad acts as 0 on the components of $\LL_{V_{\frac12}}$ which are mirror to 0-dimensional toric strata of $\ppp_{\Gv}$; alternatively, one can recall (cf. \cite[\S 1.3]{abouzaid-auroux}) that $s$ is a count of holomorphic disks, and that each of the $n$ disks it counts in $\nsec_n$ will lift to a disk in the $G$-cover.

	Finally, we note that any generic section of $\ppp_\Gv$ can be made equal to $W^\vee$ after a rescaling of its coefficients.
\end{proof}
\begin{corollary} 
	\label{cor:vl-def}
	Assuming Conjecture \ref{conj:main}, there is an equivalence of categories $\cW(V_L)\cong \MF^\Gv(\tcn,W^\vee).$
\end{corollary}
\begin{proof}This is a straightforward application of Conjecture \ref{conj:main}, bringing together our results from earlier in this section: In Proposition \ref{prop:pantscalc}, we computed that $\cW(V_L)\cong\Coh^\Gv(\tcn)$; in Proposition \ref{prop:mir-mid}, we computed $\cW(V_{\frac12}) \cong \Coh(\ppp_\Gv)$; in Lemma \ref{lem:mon-aut}, we described the clockwise monodromy automorphism $\mu^{-1}$ as tensoring with $\cO_{\ppp_\Gv}(-1),$ and in Lemma \ref{lem:generic}, we saw that the disk-counting section $s:\mu^{-1}\to \id_{\cW(V_{\frac12})}$ was the function $W^\vee.$

	We conclude that, assuming Conjecture \ref{conj:main}, we have an equivalence of categories $\cW(V_L)\cong \MF(\cW(\Vc_L), \tils) \cong \MF(\ppp_\Gv, W^\vee).$
\end{proof}

We now reach the main theorem of this paper.

\begin{theorem}\label{thm:main}
	Assuming Conjecture \ref{conj:main}, the wrapped Fukaya category $\cW(V)$ is the deformation of the 2-periodic dg-category $\Coh^{\Gv}(\pcn)_{\ZZ/2}$ by the function $W^\vee\in \CC[\pcn]^{\Gv}.$ 
	In other words, there is an equivalence of 2-periodic dg categories
  \[
		\cW(V)\cong \MF^{\Gv}(\CC^{n+1}, z_0z_1\cdots z_n+W^\vee)
  \]
	between the wrapped Fukaya category of the Milnor fiber $V$ and the category of matrix factorizations on $\CC^{n+1}$ for the function $z_0z_1\cdots z_n + W^\vee.$
\end{theorem}

\begin{proof}[Proof of Theorem \ref{thm:main}]

The proof of Theorem~\ref{thm:main} begins from the equivalence
\begin{equation}\label{eq:bcolim-undef}
\Coh^{\Gv}(\pcn)\cong \colim\left( \Coh^{\Gv}(\tpc)\gets \Coh(\ppp_\Gv)\to\Coh^{\Gv}(\{0\}) \right).
\end{equation}
proved in Lemma \ref{lem:colim-undef}.
The right-hand-side of \eqref{eq:bcolim-undef} corresponds to the colimit description of the wrapped Fukaya category $\cW(\Vc)$ given by the cover of $\Vc$ by subsectors sectors $\Vc_L,\Vc_R,$ and their intersection $\Vc_{\frac12}\times T^*\RR.$

We have seen that the Weinstein manifold $V$ has an analogous Liouville-sectorial decomposion, and in fact the right and middle subsectors $\Vc_R,\Vc_{\frac12}$ of $\Vc$ are equal to the corresponding subsectors $V_R,V_{\frac12}$ of $V$. Hence the category $\cW(V)$ admits a colimit presentation as in \eqref{eq:bcolim-undef}, but with the category $\cW(\Vc_L) \cong \Coh^{\Gv}(\tpn)$ replaced by its deformation
\begin{equation}\label{eq:vl-eqs}
\cW(V_L) \cong 
\MF^{\Gv}(\tpn,W^\vee),
\end{equation}
where the equivalence in \eqref{eq:vl-eqs} comes from Corollary \ref{cor:vl-def} (conditional on Conjecture \ref{conj:main}).

\end{proof}

%
%
%
%
%

\appendix
\section{Maslov data and 2-periodicity}
\label{sec:appendix}
In this section, we recall the data needed to define either the wrapped Fukaya category of a Liouville sector or the cosheaf of wrapped microlocal sheaves, mostly following the exposition in \cite[\S 5.3]{GPS3} and \cite[\S 10]{NS20} (to which we refer the reader interested in a more detailed discussion), and then we will explain the simplifications that occur in the 2-periodic case.

\subsubsection*{Grading data}
Let $X$ be a Weinstein manifold. Traditionally (cf. \cite[\S\S 11e-11l]{Sei-book}), the data necessary to define the Fukaya category with $\ZZ$ coefficients has been understood to be a choice of class $H^2(X;\ZZ/2)$ together with a trivialization of the complex line bundle $(\wedge_\CC^{\on{top}}TX)^{\otimes 2},$ which can be understood as a choice of class in $H^1(\LGr(X);\ZZ)$ whose restriction to each fiber $\LGr_x(X)$ represents the Maslov class.

This data can be better encapsulated, and generalized to coefficients in a general ring $R$, as follows: the stable tangent bundle of $X$ is classified by a map $X \to BU$, and the stable Lagrangian Grassmannian $\LGr(X)$ is classified by the composition 
$X\to BU \to B(U/O) = B^2(\ZZ\times BO).$ 
The delooped J-homomorphism gives a map 
$B^2(\ZZ\times BO)\xrightarrow{B^2J}B^2\Pic(\SS)$
to the delooping of the spectrum of invertible modules for the sphere spectrum $\SS.$ For any ring $R$, there is a map $\Pic(\SS)\to \Pic(R)$ induced by the map $\SS\to \RR.$
\begin{definition}
	{\em Grading/orientation data} for $X$ with $R$ coefficients is a trivialization of the composition
	\begin{equation}\label{eq:grading-data}
		X\xrightarrow{\LGr} B(U/O)=B^2(\ZZ\times BO)\xrightarrow{B^2J} B^2\Pic(\SS)\to B^2\Pic(R).
	\end{equation}
\end{definition}

When $R=\ZZ,$ the space of invertible $\ZZ$-modules is $\Pic(\ZZ) =\ZZ\times B\ZZ/2$: there are $\ZZ$ distinct classes of invertible $\ZZ$-modules, namely the homological shifts $\ZZ[n]$ of the rank 1 free module, and they have automorphism group $\ZZ/2$ generated by multiplication by $-1$.

Grading/orientation data with $\ZZ$ coefficients is therefore given by a nulhomotopy of the map
\[
	X\to B^2\ZZ\times B^3(\ZZ/2),
\]
where it can be shown that the first factor (``grading data'') classifies $(\wedge^{\on{top}}_\CC TX)^{\otimes 2}$ and the second factor (``orientation data'') admits a canonical trivialization (giving a correspondence between choices of nulhomotopy for the second factor and maps $X\to B^2(\ZZ/2)$). We therefore see that grading/orientation data for $\ZZ$ coefficients agrees with the traditional data used to define the Fukaya category.

In practice, there is a universal way of constructing grading/orientation data for a symplectic manifold.
\begin{definition}
	A {\em stable polarization} of $X$ is a trivialization of the map $X\to B(U/O)$ classifying the Lagrangian Grassmannian bundle $\LGr(X).$
\end{definition}
It is clear that a stable polarization induces grading/orientation data for any ring $R$, since a trivialization of $X\to B(U/O)$ trivializes the whole composition \eqref{eq:grading-data}.

As $\LGr=U/O$ is an infinite loop space, a trivialization of the stable Lagrangian Grassmannian bundle $\LGr(X)$ is the same as a section of it, or equivalently a section $\sigma$ of the Lagrangian Grassmannian of the stable symplectic normal bundle of $X$ (which is the negative of the stable tangent bundle). 

In \cite{Sh-hp,NS20} it is explained that if $X$ is a Weinstein manifold with skeleton $\LL,$ then $X$ admits a possibly high-codimension embedding into a cosphere bundle $S^*M,$ and the data of a section $\sigma$ as above is precisely the data necessary to thicken $\LL$ to a Legendrian $\LL^\sigma$ in $S^*M$ and therefore to define microlocal sheaves on $\LL^\sigma$. In \cite{GPS3} it is shown that the category so defined agrees (up to passage to opposite categories) with the wrapped Fukaya category of $X$, defined using grading/orientation data coming from the stable polarization $\sigma.$ These are the results we have summarized as Theorems \ref{thm:global-mush} and \ref{thm:gps-maincomparison} above.

\subsection{Gluing Fukaya categories}
At various points in this paper, we will study a Weinstein manifold or sector $X$ whose skeleton $\LL$ admits a cover $\LL = U_1\cup U_2$ by two open sets $U_i$ intersecting in $U_{12}:=U_1\cap U_2,$ and we would like to present the wrapped Fukaya category of $X$ as a gluing of the categories $\wmsh(U_1),\wmsh(U_2)$ along $\wmsh(U_{12}).$ We may compute each of these categories with locally chosen grading/orientation data $\sigma_1,\sigma_2,\sigma_{12}$, but in order to ensure that these categories glue into a global Fukaya category of $X$, we must show that the restrictions of $\sigma_i$ to $U_{12}$ are both homotopic to $\sigma_{12},$ so that, up to homotopy, this data glues into a global choice of grading/orientation data on $X$.

All of the symplectic manifolds studied in this paper are complete intersections in $\CC^n$ or $(\CC^\times)^n,$ and their grading/orientation data is induced from the standard polarizations of $\CC^n = T^*\RR^n$ and $(\CC^\times)^n = T^*((S^1)^n).$ To establish that our gluings of Fukaya categories are sensible, it is therefore sufficient to check that grading/orientation data coming from these polarizations agree, which we will now do. (We will see that making the grading data for these two polarizations agree will require altering our coefficient ring.)

\subsubsection*{Orientation data}
As described above, grading/orientation data with integral coefficients is a nulhomotopy of the map
\[
	X\to B^2\ZZ \times B^3(\ZZ/2),
\]
where the second factor admits a canonical trivialization.
A stable polarization $\sigma$ is one way to trivialize this map, but the trivialization on the second factor may not agree with the canonical one.
Indeed, as explained in \cite[Lemma 3.9]{CKNS}, the difference between these two trivializations is measured by the second Stiefel-Whitney class $w_2(\sigma).$

If $X=T^*M$ is a cotangent bundle and $\sigma$ is its cotangent fiber polarization, this means that the orientation data determined by $\sigma$ differs from the canonical trivialization of the map $X\to B^3(\ZZ/2)$ if and only if $w_2(M)$ is nonzero. We conclude that the cotangent fiber trivialization on $\CC^n$ or on $(\CC^\times)^n$ induces the canonical orientation data.

\subsubsection*{Grading data and 2-periodicity}

In contrast to the situation for orientation data, the grading data (with integral coefficients) for the cotangent fiber polarizations on $\CC^n$ and $(\CC^\times)^n$ does not agree. The discrepancy between the two trivializations of the map to $B^2\ZZ$ is measured by a map to $\Omega B^2\ZZ= B\ZZ.$
As a map $\CC^\times \to B\ZZ,$ this is homotopic to the map $S^1\to S^1$ of degree 2, corresponding to the fact that traversing the Maslov cycle in $\LGr$ acts by the automorphism $[2]$ of degree-shift by 2.


However, the discrepancy vanishes if we work with 
coefficients in $\ZZ((\beta)),$ where $\beta$ is a variable of homological degree 2.
Grading/orientation data for $\ZZ((\beta))$ is a trivialization of the composite map
\[
	X\to B^2(\Pic(\ZZ)) = B^2(\ZZ\times B\ZZ/2)\to B^2(\ZZ/2\times B\ZZ/2) = \Pic(\ZZ((\beta))),
\]
where the quotient $\ZZ\to \ZZ/2$ reflects the fact that the $n$-shifted invertible modules $\ZZ((\beta))[n]$ of the same parity are all isomorphic to each other. 

We conclude that the cotangent fiber polarizations of $\CC$ and $\CC^\times$ induce the same grading data with $\ZZ((\beta))$ coefficients. Dg-categories linear over $\ZZ((\beta))$ may equivalently be thought of as 2-periodic ($\ZZ$-linear) dg-categories, with 2-periodicity element $\beta.$ In other words, we have shown that so long as we work with 2-periodic dg-categories, all the restrictions of the grading and orientation data used in this paper agree up to homotopy.

\bibliographystyle{alpha}
\bibliography{refs}

\end{document}